\documentclass[reqno]{amsart}
\usepackage{amssymb, amsmath}
\usepackage{natbib}
\usepackage{lscape}
\usepackage{dsfont}
\usepackage{mathrsfs}
\usepackage{bm}

\usepackage[pdftex,plainpages=false,colorlinks,hyperindex,bookmarksopen,linkcolor=red,citecolor=blue,urlcolor=blue]{hyperref}

\DeclareMathAlphabet{\mathpzc}{OT1}{pzc}{m}{it}

\bibpunct{[}{]}{;}{n}{,}{,}

\newtheorem{te}{Theorem}[section]

\newtheorem{os}[te]{Remark}
\newtheorem{prop}[te]{Proposition}
\newtheorem{lem}[te]{Lemma}
\newtheorem{ex}[te]{Example}
\newtheorem{coro}[te]{Corollary}
\numberwithin{equation}{section}

\allowdisplaybreaks

\newcommand{\br}[1]{{\textcolor{black}{#1}}}
\newcommand{\brr}[1]{{\textcolor{black}{#1}}}
\newcommand{\brbr}[1]{{\textcolor{black}{#1}}}
\def \l { \left( }
\def \r {\right) }
\def \ll { \left\lbrace }
\def \rr { \right\rbrace }

\begin{document}

	\title[]{On semi-Markov processes and their Kolmogorov's integro-differential equations}
	\author[E. Orsingher]{Enzo Orsingher$^{1}$}
	\author[C. Ricciuti]{Costantino Ricciuti$^{1,\text{\textsection}}$}
	\address[1]{Dipartimento di Scienze Statistiche, Sapienza - Universit\`{a} di Roma}
	\author[B. Toaldo]{Bruno Toaldo$^{2}$}
	\address[2]{Dipartimento di Matematica e Applicazioni ``Renato Caccioppoli'' - Universit\`{a} degli Studi di Napoli Federico II}
	\email[\textsection  Corresponding Author]{costantino.ricciuti@uniroma1.it}
	\date{\today}

		\begin{abstract}
Semi-Markov processes are a generalization of Markov processes since the exponential distribution of time intervals is replaced with an arbitrary distribution. This paper provides an integro-differential form of the Kolmogorov's backward equations for a large class of homogeneous semi-Markov processes, having the form of an abstract Volterra integro-differential equation. An equivalent evolutionary (differential) form of the equations is also provided. Fractional equations in the time variable are a particular case of our analysis. Weak limits of semi-Markov processes are also considered and their corresponding integro-differential Kolmogorov's equations are identified.
\end{abstract}

	\maketitle
\tableofcontents


\section{Introduction}
Semi-Markov processes have been introduced by L\'evy \cite{levysemi} and Smith \cite{smith} in order to reduce the limitation induced by the exponential distribution of the corresponding time intervals. In these papers jump semi-Markov processes were considered, i.e., jump processes with a waiting time between jumps which is not necessarily given by an exponential random variable. This is the immediate generalization of Markov chains since the Markov property is the typical consequence of the lack of memory of the exponential distribution. The general theory of semi-Markov processes has then been developed by Pyke \cite{pykefinite, pykeinfinite}. The first generalization of the Kolmogorov's equations to the semi-Markov case was given in Feller \cite{fellersemi}. In this paper the author provided an integral form for the backward equation when a semi-Markov process runs on a countable state-space, but from the discussion in \cite{fellersemi} it is clear that the generalization of such equations to any state space was not so far. In succesive years, indeed, the theory of semi-Markov processes has been developed (e.g. Cinlar \cite{cinlarsemi, cinlar}, Gihman and Skorohod \cite{gihman}, Jacod \cite{jacod}). Recent developments of the theory can be found in Harlamov \cite{harlamov} and in Korolyuk and Swishchuk \cite{koro} in which semi-Markov processes are discussed in full generality.

A large class of semi-Markov processes can be equivalently constructed as time-changed Markov processes. This fact is well formalized in Kurtz \cite{kurtz} in the case where the waiting times between jumps have finite mean. A more general approach is proposed in Kaspi and Maisonneuve \cite{kaspi} by assuming a Markov additive process $(A_t, D_t)$ and defining $X(t) = A(L(t))$ where $L(t)$ is the hitting time process of $D_t$. 
In the present paper we consider indeed semi-Markov processes which can be obtained as a time-changed Markov process. In recent years Baeumer and Meerschaert \cite{fracCauchy}, Meerschaert and Scheffler \cite{meertri}, Meerschaert et al. \cite{meerpoisson} considered Markov processes time-changed via an independent inverse of an $\alpha$-stable subordinator and shown that these processes are governed by time-fractional equations, which are very popular in applications (e.g. \cite{meerbook, metzler} for a review of possible applications or \cite{hairer} for a recent development; see \cite{grot1, grot2} for a different probabilistc approach related to fractional diffusion equation). When the same thing is done with a more general independent inverse subordinator then the equations become more general integro-differential equations (Kochubei \cite{kochu}, Kolokoltsov \cite{kololast, KoloCTRW, kolokoltsov}, Toaldo \cite{toaldopota, toaldodo}) or pseudo-differential in the time-variable \cite{meertri}. These processes can be often viewed as semi-Markov processes as discussed in Meerschaert and Straka \cite{meerstra}. Hence this suggests that there is a strong relationship between integro-differential equations and the Kolmogorov's equation of semi-Markov processes. In the present paper such a relationship is clearly established in a very general framework. The Kolmogorov's equations are firstly investigated when the semi-Markov processes have stepped paths. Such equations turn out to be Volterra integro-differential equations. Our framework also includes the case in which the Markov process is non independent on the random time process. This yields to variable order Volterra integro-differential equations having the form
\begin{align}
\frac{d}{dt} \int_0^t q(s, \cdot) \, k(t-s,\cdot) \, ds \, - k(t,\cdot) q(0,\cdot)\, = \, \l G q(t)\r (\cdot)
\label{volterraintro}
\end{align}
for $q:[0, \infty) \mapsto \mathfrak{B}$ where $\mathfrak{B}$ is a Banach space, $k(t, \cdot)$ is a suitable convolution kernel and $G$ is the generator of the Markov process. When the kernel $k$ does not depend on the vector variable of the Banach space the regularity of the equation has been recently investigated by \cite{caffarelli}. Further, weak limits of stepped semi-Markov processes are also studied and the corresponding Kolmogorov's equations are determined.
Time-fractional equations can be always viewed as particular and interesting cases of the equations studied in this paper.

\section{Preliminaries}
\br{We collect in this section some technical information which will be used throughout the paper.}
\subsection{Complete monotonicity and Bernstein functions} We recall here some basic facts on Bernstein functions and complete monotonicity. We refer to \cite{librobern} for such information. A function $f:(0, \infty) \mapsto \mathbb{R}$ is said to be a Bernstein function if it is of class $C^\infty$, $f(\lambda) \geq 0$ and $(-1)^{n-1}f^{(n)}(\lambda) \geq 0$ for all $\lambda >0$ (\cite[Definition 3.1]{librobern}). For a non-negative $C^\infty$ function $f$ to be a Bernstein function it is necessary and sufficient that $f^\prime$ is completely monotone. A function $g:(0, \infty) \mapsto \mathbb{R}$ is completely monotone if it is of class $C^\infty$ and such that $(-1)^{n}g^{(n)}(\lambda) \geq 0$ for all $n \in \mathbb{N} \cup \ll 0 \rr$ and $\lambda >0$; hence (e.g. \cite[Theorem 1.4]{librobern}) it can be written as the Laplace transform of a unique measure $K$ on $[0, \infty)$, i.e., for all $\lambda >0$,
\begin{align}
g(\lambda) \, = \, \int_{[0, \infty)}e^{-\lambda s}K(ds).
\end{align}
A consequence of the above facts is that a function $f$ is a Bernstein function if, and only if, it can be written in the form \cite[Theorem 3.2]{librobern}
\begin{align}
f(\lambda) \, = \, a+b\lambda + \int_0^\infty \l 1-e^{-\lambda s} \r \nu(ds)
\end{align}
where $a$ and $b$ are non-negative constants and $\nu(\cdot)$ is a measure  on $(0, \infty)$ satisfying the integrability condition $\int_0^\infty \l s \wedge 1 \r \nu(ds) < \infty$. A subclass of Bernstein functions, playing a central role in this paper is the class of complete Bernstein functions. A Bernstein function is said to be complete if the corresponding L\'evy measure has a completely monotone density with respect to the Lebesgue measure \cite[Definition 6.1]{librobern}. Furthermore, a Bernstein function $f \neq 0$, is complete if, and only if, can be represented as $f = 1/h$ where $h$ is a (non-negative) Stieltjes function, i.e., $h$ admits the representation
\begin{align}
h(\lambda) \, = \, \frac{a}{\lambda}+b+\int_0^\infty \frac{1}{s+\lambda} m(ds)
\end{align}
where $m$ is a measure on $(0, \infty)$ such that
\begin{align}
\int_0^\infty \frac{1}{1+s}m(ds) < \infty.
\end{align}

\subsection{Subordinators and non-homogeneous subordinators}
\label{nonhomosub}
A subordinator is a non-decreasing L\'evy process. Every Bernstein function is the Laplace exponent of a subordinator \cite[Theorem 1.2]{bertoins}. Hence with the symbol $\sigma^f(t)$ we denote the subordinator with transition probabilites $\mu_t(ds)$ such that
\begin{align}
\int_0^\infty e^{-\lambda s} \mu_t(ds) \, = \, \mathds{E}e^{-\lambda \sigma^f(t)} \, = \, e^{-tf(\lambda)}.
\end{align}
A non-homogeneous subordinator, say $\sigma^\Pi (t)$, $t \geq 0$, is a non-decreasing additive process in the sense of Sato \cite[Definition 1.6]{satolevy}. Hence it is a non-decreasing process with independent increments, stochastically continuous and with a.s. c\`{a}dl\`{a}g paths and such that $\sigma^\Pi (0) = 0$ a.s. Such a process can be constructed as proposed in \cite[Sec. 2]{orsrictoapota}, i.e., 
\begin{align}
\sigma^\Pi (t) \, : = \, b(t) + \sum_{0 \leq s \leq t} \mathpzc{e}(s)
\label{sumpoisson}
\end{align}
where $\mathpzc{e}(s)$ is a Poisson point process in $\mathbb{R}^+$ whose characteristic measure $\nu(dx, dt)$ on $(0, \infty) \times [0, \infty)$ satisfies the integrability condition
\begin{align}
\int_{(0, \infty)\times[0,t]} (x\wedge 1) \nu(dx, dw) \, < \, \infty.
\end{align}
Hence if $\nu(dx, dt) = v(dx) dt$, where $dt$ is the Lebesgue measure, the construction of a subordinator is obtained. In what follows we will always assume, as in \cite{orsrictoapota}, that $\nu(ds, \cdot)$ is absolutely continuous with respect to the Lebesgue measure and we denote a density as $\nu(ds, t) dt : = \nu(ds, dt)$. Furthermore we will assume that the process has no drift, hence in \eqref{sumpoisson} we have $b(t) =0$, for all $t \geq 0$. Therefore the expected number of jumps from $x$ to $x+dx$ occuring up to $t$ is given by
\begin{align}
\phi(dx, t) \, : = \, \int_0^t \nu(dx, s) ds.
\end{align}
By \cite[Eqs (2.6) and (2.7)]{orsrictoapota} we have
\begin{align}
\mathds{E} e^{-\lambda \sigma^\Pi (t)} \, = \, e^{-\Pi (\lambda, t)}
\end{align}
where
\begin{align}
\lambda \mapsto \Pi (\lambda, t) \, : = \, \int_0^\infty \l 1-e^{-\lambda s} \r \phi(ds, t)
\end{align}
or equivalently, by \cite[Eq. (2.12)]{orsrictoapota},
\begin{align}
\Pi (\lambda, t) \, = \, \int_0^t f(\lambda, s) ds
\end{align}
where $\lambda \mapsto f(\lambda, s)$ is, for each $s$, the Bernstein function
\begin{align}
f(\lambda, s) \, = \, \int_0^\infty \l 1-e^{-\lambda w} \r \nu(dw, s).
\label{29}
\end{align}
We will further assume throughout the paper that $\nu$ is such that $\sup_s f(\lambda, s) < \infty$ and that $s \mapsto f(\lambda, s)$ is continuous.

\section{Stepped semi-Markov processes and time-changed Markov processes}
\subsection{Stepped semi-Markov processes}
\label{defsemi}
Let $\l \mathcal{S}, \mathfrak{S} \r$ be a state space where $\mathfrak{S}$ is a $\sigma$-algebra on $\mathcal{S}$, generated from the space $\mathcal{S}$ endowed with the discrete topology induced by the metric
\begin{align}
\rho (x,y) \, = \, \begin{cases} 1, \qquad & x \neq y, \\ 0, & x=y. \end{cases}
\end{align}
We consider on $\mathcal{S}$ right-continuous processes $X(t)$, $t \geq 0$, whose paths are stepped functions. Hence the paths are functions $t \mapsto x(t)$, such that for any $t \geq 0$ there exists a $\delta >0$ such that for all $h \in (0, \delta)$ it is true that $x(t)=x(t+h)$, i.e. the functions $t \mapsto x(t)$ are right-continuous in the discrete topology, and further they have a finite number of discontinuities on any finite interval (of time). Processes with these paths are semi-Markov processes in the sense of Gihman and Skorohod (see \cite[Chapter 3]{gihman} and also \cite[Chapter 3, Section 12, p. 76]{harlamov}) if the couple $\l X(t), \gamma^X(t) \r$, where
\begin{align}
\gamma^X(t) \, := \,  t - \l 0 \vee \sup \ll s \leq t : X(s) \neq X(t) \rr \r,
\end{align}
is a strict Markov process. A process with these properties can be constructed as follows (see \cite[Chapter 3, Section 3]{gihman}). Assume a family of probability spaces $\l \Omega, \mathcal{F}, P^x \r$, $x \in \mathcal{S}$. Let $X_n$ be a discrete-time Markov chain on $\mathcal{S}$ under $P^x$ such that $P^x \l X_0 = x \r =1$ and define for all $B \in \mathfrak{S}$, the transition probabilities $B \mapsto h(x,B) = P^x \l X_1 \in B \r$. In the case $\mathcal{S}$ is countable we denote by $\l H \r_{ij} = h_{ij}$ the transition matrix of $X_n$. Let $\eta_i$ be a sequence of i.i.d. r.v.'s jointly independent from $X_n$ under any $P^x$ each one of which is uniformly distributed on $[0,1]$. Let $F_{x,y}$ be a c.d.f. of a non-negative r.v. for any $x,y \in \mathcal{S}$ and define a function $\varphi_{x,y}:[0,1] \mapsto [0, \infty)$ such that $\varphi_{x,y} (\eta_i)$ have distributions $F_{x,y}$ for any $i$.
Then let $J_i = \varphi_{X_i,X_{i+1}} (\eta_i)$ and define
\begin{align}
X(t) \, = \, X_n, \qquad \sum_{i=0}^{n-1} J_i \leq t < \sum_{i=0}^{n}J_i.
\label{defx}
\end{align}
This is equivalent to say that for any $x,y,z \in \mathcal{S}$
\begin{align}
P^x \l J_n \leq t \mid X_n = z, X_{n+1}=y \r \, = \, P^z \l J_0 \leq t \mid X_1 = y \r \, = \,  F_{z,y}(t).
\end{align}
We denote $\overline{F}_{x,y}(t) := 1-F_{x,y}(t)$. When $F_{x,y}$ has a density with respect to the Lebesgue measure, we denote it by $\mathpzc{f}_{x,y}(t)$. It is clear that if $\overline{F}_{x,y}(t) = e^{-\theta(x) t}$ then we have that \eqref{defx} defines a continuous-time Markov process. In this case we use the symbol $\mathcal{J}_i$ to denote such exponential r.v.'s, the symbol $M(t)$, $t \geq 0$, to denote the corresponding Markov process and $t \mapsto m(t)$ for the sample paths of $M(t)$. 
Now let $T_n = \sum_{i=1}^{n-1} J_i$ and define
\begin{align}
N^J(t) \, := \, \max \ll n \in \mathbb{N}: T_n \leq t \rr.
\label{ren}
\end{align}
By using \eqref{ren} we can say that \eqref{defx} is equivalent to define
\begin{align}
X(t) \, := \, X_{N^J(t)}.
\label{26}
\end{align}
It is clear that for $M(t)$ the process $N^J$ reduces to a birth process. In this case we denote the epochs of the birth process with the symbol $\mathcal{T}_n$. From the discussion above it is therefore clear that $P^x \l \mathcal{T}_{n+1}- \mathcal{T}_n > t \mid X_n = y \r = e^{-\theta(y)t}$.

Finally, note that in view of \cite[Chapter 3, Section 3, Lemma 2]{gihman} the process $X(t)$ satisfies the defining properties of semi-Markov processes, i.e., the couple $\l X(t), \gamma^X(t) \r$, where $\gamma^X (t) := t-T_{N^J(t)}$, is a strict (homogeneous) Markov process.
Hence it is clear that our semi-Markov processes are time-homogeneous in the sense that
\begin{align}
P^x \l X(t+\tau) \in B \mid X(\tau) = y, \gamma^X(\tau) =s \r \, = \, P^y \l X(t) \in B \mid \gamma^X (0) = s \r,
\label{homo}
\end{align}
for any $t,\tau \geq 0$, $0 \leq s \leq \tau$, $x,y \in \mathcal{S}$ and $B \in \mathfrak{S}$.

\subsection{Construction as time-changed Markov processes}
\label{sectimech}
The aim of this section is to make precise the following heuristical remarks.
The results below show when a homogeneous semi-Markov process is given by a random time change of a Markov process and which type of time-change is allowed. Suppose that at the random times $\l \mathcal{T}_i \r_{i \in \mathbb{N}}$ a Markov process $M(t^\prime)$ jumps in the states $\l X_i \r_{i \in \mathbb{N}}$ and a semi-Markov process $X(t)$ jumps in the same states $\l X_i \r_{i \in \mathbb{N}}$ at the random times $(T_i)_{i \in \mathbb{N}}$. If $t = g(t^\prime)$ we have that $T_i = g (\mathcal{T}_i)$, in distribution, and $X(t) = M(t^\prime) = M(g^{-1}(t))$. This can happen for a suitably defined strictly increasing function $t \mapsto g(t)$ which must be meant as  either a deterministic or as a random function. To determine which function $g$ is allowed, consider that the r.v.'s $g \l \mathcal{T}_{i+1}\r- g \l \mathcal{T}_i \r$ must be independent r.v.'s. \brr{Note also that the behaviour of the function $g$ can be dependent on the position of the process, i.e., the increment $g(t)-g(s)$ depends on the r.v.'s $M(w),s \leq w \leq t $. \brbr{This means that the (infinitesimal) increment $g(s+ds)-g(s)$ is given by the increment of a different function $\sigma^x(s+ds)-\sigma^x(s)$, conditionally on $M(s) = x$. Furthermore since $M(t)$ is a stepped process, we can write
\begin{align}
g(t) = g(\mathcal{T}_i)+\sigma^{X_i}(t)-\sigma^{X_i}(\mathcal{T}_i), \qquad \mathcal{T}_i \leq t < \mathcal{T}_{i+1},
\end{align}
since the position $X_i$ reached at $\mathcal{T}_i$ is mantained up to time $\mathcal{T}_{i+1}$}.}
\brbr{The holding times $T_{i+1}-T_i$ are the r.v.'s
\begin{align}
g(\mathcal{T}_{i+1}-) - g(\mathcal{T}_i) = \sigma^{X_i}(\mathcal{T}_{i+1}-)-\sigma^{X_i}(\mathcal{T}_i),
\label{39}
\end{align}
and to preserve time-homogeneity we may require that the distribution of $T_{i+1}-T_i$ does not depend on $i$ which is the number of the jumps. Hence if the increment satisfies $\sigma^{X_i}(t)-\sigma^{X_i}(s)=\sigma^{X_i}(t-s)$, in distribution, and if $\sigma^{X_i} (t)=\sigma^{X_i}(t-)$, then the distribution of $T_{i+1}-T_i$, which is the distribution of \eqref{39}, is independent on $i$ since  $\mathcal{T}_{i+1}-\mathcal{T}_i$ is an exponential r.v. with parameter $\theta (x)$ (conditionally on $X_i=x)$.} This suggest to choose $\sigma^{X_i}(t)$ as a strictly increasing process with stationary and independent increments, i.e., a strictly increasing subordinator. Since any increment $\sigma^{X_i}(t)$ is strictly increasing and left-invertible then $g$ is strictly increasing and left-invertible and therefore
\begin{align}
X(t) = X_n, \qquad g(\mathcal{T}_n) \leq t < g(\mathcal{T}_{n+1}),
\end{align}
is equivalent to
\begin{align}
X(t) = X_n, \qquad \mathcal{T}_n \leq g^{-1}(t) < \mathcal{T}_{n+1}.
\end{align}
In particular since $g$ is strictly increasing then the left-inverse $g^{-1}$ is the hitting-time process of $g$.
This heuristically shows that in this case $X(t)$ and $M(g^{-1}(t))$ are the same process and since in the previous remarks the increments of $g$ depend on the position $X_i$ (and on any $\tau_i$) then $g(t)$ is dependent on $M(t)$. This will be made precise below.

We show here that processes $X(t)$ defined as in Section \ref{defsemi} can be constructed under suitable assumptions as a particular time-change of $M$.
Consider a family of Bernstein functions $\ll f(\lambda, x,y) \rr_{(x,y) \in \mathcal{S}\times \mathcal{S}}$ having representation
\begin{align}
f(\lambda, x,y ) \, = \, \int_0^\infty \l 1-e^{-\lambda s} \r \nu(ds, x,y),
\end{align}
and let $\nu((0, \infty),x,y)=\infty$ for each $(x,y) \in \mathcal{S}\times \mathcal{S}$.
Let $\ll \sigma^{(x,y)}(t) \rr_{(x,y) \in \mathcal{S}\times \mathcal{S}}$, $t \geq 0$, be a family of subordinators with Laplace exponent $f(\lambda, x,y)$. Then consider $\ll L^{(x,y)} (t) \rr_{(x,y) \in \mathcal{S}\times \mathcal{S}}$ which is a family of inverses of $\sigma^{(x,y)}$.
Now consider a non-homogeneous subordinator defined as in \eqref{sumpoisson} with $b(t) =0$ for all $t \geq 0$ and with characteristic measure
\begin{align}
v(ds,t) \, = \, \sum_{k=1}^\infty \nu(ds, x_k, x_{k+1}) \, \mathds{1}_{\left[ \tau_k,\tau_{k+1} \right]}(t)
\label{lev}
\end{align}
where \br{the symbols $x_k$ and $\tau_k$ indicate the realization of $X_k$ and $\mathcal{T}_k$} so that
\begin{align}
t \mapsto m(t) = x_k, \qquad \tau_k \leq t < \tau_{k+1}.
\label{mpiccolodef}
\end{align}
Equivalently if we denote by $t \mapsto n(t)$ the sample paths of the counting process $N(t)$ associated with $M(t)$, we can rewrite \eqref{mpiccolodef} as
\begin{align}
m(t) = x_{n(t)}.
\end{align}
This is to say that each fixed path $t \mapsto m(t)$ of $M(t)$, whose discontinuities (jumps) are in the points $t = \tau_k$, defines a different measure $v(ds, t)dt$.
Hence the increments of $\sigma^\Pi(t)$ are given by
\begin{align}
\sigma^\Pi(t) \, = \, \sigma^\Pi (\mathcal{T}_n)+ \sigma^{(X_n, X_{n+1})}(t-\mathcal{T}_n), \qquad \mathcal{T}_n \leq t < \mathcal{T}_{n+1}.
\label{incre}
\end{align}
The process $\sigma^\Pi(t)$ is therefore a different non-homogeneous subordinator for any fixed path of $m(t)$ with Laplace exponent $\Pi (\lambda, t)$ which is determined by the path of $M(t)$. However the r.v. $\sigma^\Pi(t)$ depends only on the events needed at time $t$, i.e., we have
\begin{align}
 \mathds{E}^x\left[e^{-\lambda \sigma^\Pi (t)} \mid \tau_0, \cdots, \tau_{n(t)}, x_0, \cdots, x_{n(t)+1} \right] \notag \, = \, e^{-\Pi (\lambda, t)}
 \end{align}
 where
\begin{align}
\Pi (\lambda, t) \, = \,&  \int_0^t \int_0^\infty \l 1-e^{-\lambda s} \r v(ds,y)dy \ \notag \\
= \, &  \int_0^t \int_0^\infty \l 1-e^{-\lambda s} \r \sum_{k=0}^{n(t)} \nu(ds,x_k,x_{k+1}) \, \mathds{1}_{\left[ \tau_k,\tau_{k+1} \right]}(y)dy  \notag \\
= \, &   \sum_{k=0}^{n(t)-1} (\tau_{k+1}-\tau_{k} ) f(\lambda, x_k, x_{k+1}) + (t-\tau_{n(t)})f(\lambda, x_{n(t)}, x_{n(t)+1}),
\label{lapl}
\end{align}
where $\tau_0=0$ and, under $P^x$, $x_0=x$.
Then we define
\begin{align}
L^\Pi (t) \, := \, \inf \ll s \geq 0 : \sigma^\Pi (s) \geq t \rr.
\end{align}
\br{Since it is assumed that $\nu((0, \infty), x, y) = \infty$ for any $(x,y)$ then $v((0, \infty), t) = \infty$ for any $t$ and thus the process $\sigma^\Pi (t)$ is strictly increasing on any finite interval of time \cite[Proposition 2.2]{orsrictoapota}. Hence $L^\Pi \l \sigma^\Pi (t) \r = t$, a.s.}
The following Theorem is the precise statement of the heuristical discussion at the beginning of this section.
\begin{te}
\label{tetimech}
Let $X(t)$ be a stepped semi-Markov process (as in \eqref{defsemi}). Assume $\nu((0, \infty),x,y))=\infty$ for any $(x,y) \in \mathcal{S} \times \mathcal{S}$. The following assertions are equivalent.
\begin{enumerate}
\item For any $(x,y) \in \mathcal{S} \times \mathcal{S}$ and $z \in \mathcal{S}$, $\overline{F}_{x,y}(t) = \mathds{E}^ze^{-\theta(x)L^{(x,y)}(t)}$.
\item $X(t)$ and $M \l L^\Pi (t) \r$ are the same process.
\end{enumerate}
\end{te}
\begin{proof}
The process $\sigma^\Pi(t)$ is strictly increasing on any finite interval (\cite[Proposition 2.2]{orsrictoapota}) then $L^\Pi(t)$ has continuous sample paths. Since $X(t)$ and $M \l L^\Pi(t) \r$ have the same embedded chain, to prove that $X(t)$ coincides with $M(L^\Pi(t))$ it is sufficient to prove that $M(L^\Pi(t))$ has the same waiting times between the jumps of $X(t)$. Note that
\begin{align}
M \l L^\Pi(t)  \r \, = \, X_n, \qquad \mathcal{T}_n \leq L^\Pi(t) < \mathcal{T}_{n+1},
\end{align}
and hence by using \cite[Lemma 2.1]{meerpoisson} we have that the epochs of $M(L^\Pi(t))$ occurr at the random times $\sigma^\Pi (\mathcal{T}_n-)$. But it is true that, a.s., $\sigma^\Pi(\mathcal{T}_n-) = \sigma^\Pi (\mathcal{T}_n)$ (\cite[Theorem 2.1]{orsrictoapota}) and hence we have by \eqref{incre} that $\sigma^\Pi(\mathcal{T}_{n+1}) - \sigma^\Pi (\mathcal{T}_n) = \sigma^{(X_n,X_{n+1})} \l \mathcal{J}_n \r$ (we recall that $\mathcal{J}_n = \mathcal{T}_{n+1}-\mathcal{T}_n$). Since $\mathcal{J}_n$ are the holding times of the process $M(t)$, conditionally on $X_i=x$, are independent exponential r.v.'s with parameter $\theta(x)$. Hence we can compute the Laplace transform of the r.v. $\sigma^{(X_n,X_{n+1})} \l \mathcal{J}_n \r$ as
\begin{align}
\mathds{E}^x \l e^{-\lambda \sigma^{(X_n, X_{n+1})} \l \mathcal{J}_n \r} \mid X_n = y, X_{n+1}=z \r  \, = \, &  \mathds{E}^x e^{-\mathcal{J}_nf(\lambda, y,z)} \notag \\ = \, & \int_0^\infty \theta(y) e^{-\theta(y) w} e^{-wf(\lambda, y,z)} dw \notag \\ 
= \, & \frac{\theta(y)}{\theta(y)+f(\lambda, y,z)}.
\end{align}
Since by \cite[Corollary 3.5]{meertri} we have that
\begin{align}
\int_0^\infty e^{-\lambda t} \mathds{E}^x e^{-\theta(y) L^{(y,z)}(t) } dt \, = \,\lambda^{-1} \frac{f(\lambda, y, z)}{\theta(y)+f(\lambda, y,z)},
\label{319}
\end{align}
then the corresponding density has Laplace transform
\begin{align}
\int_0^\infty e^{-\lambda t} \mathpzc{f}_{y,z}(dt) \, = \, \frac{\theta(y)}{\theta(y)+f(\lambda, y,z)}.
\label{3200}
\end{align}
Therefore, for $\lambda >0$,
\begin{align}
\int_0^\infty e^{-\lambda t}\mathpzc{f}_{y,z} (dt) =\mathds{E}^x \l e^{-\lambda \sigma^{(X_n, X_{n+1})} \l \mathcal{J}_n \r} \mid X_n = y, X_{n+1}=z \r 
\end{align}
and therefore the proof is complete.
\end{proof}
\subsection{Complete monotonicity of the survival function} 
From the discussion above we know that a stepped semi-Markov process can be viewed as the random time-change (independent or dependent) of a Markov process. For example in \cite{jacod} or \cite{kurtz} it is also discussed the existence of a time-change relationship between Markov and semi-Markov processes. We provide here a condition which guarantees the existence of a time-change relationship between $M$ and $X$ in the form described in Theorem \ref{tetimech}. This condition also identify the process $\sigma^\Pi(t)$ when no information are given. Hence the following result is also a tool to identify the \br{random time process} related to a semi-Markov process.
\begin{te}
\label{tecm}
Let $K_{x,y}(\cdot)$, $(x,y) \in \mathcal{S} \times \mathcal{S}$, be a family of measures on $[0, \infty)$ which satisfies $K_{x,y}[0, \infty) =1$ and $\int_0^\infty s K_{x,y}(ds) = \infty$. The following assertions are equivalent.
\begin{enumerate}
\item For any $(x,y) \in \mathcal{S} \times \mathcal{S}$ it is true that the functions $t \mapsto \overline{F}_{x,y}(t)$ are completely monotone functions with respect to the measures $K_{x,y}(\cdot)$.
\item There exists a process $L^\Pi (t)$ such that $X(t)$ and $M \l L^\Pi (t) \r$ are the same process defined by setting in \eqref{lapl}
\begin{align}
\Pi(\lambda, t)  \, = \, \sum_{k=0}^{n(t)-1}(\tau_{k+1}-\tau_k) f(\lambda, x_k, x_{k+1})+(t-\tau_{n(t)}) f(\lambda, x_{n(t)}, x_{n(t)+1}),
\label{reprf}
\end{align}
with
\begin{align}
f(\lambda, x,y) \, = \,\br{\theta(x)} \frac{\lambda \widetilde{\overline{F}}_{x,y}(\lambda)}{1-\lambda \widetilde{\overline{F}}_{x,y}(\lambda)} \, = \, \br{\theta(x)}\int_0^\infty \l 1-e^{-\lambda s} \r \nu(ds,x,y)
\label{320}
\end{align}
which are unbounded complete Bernstein functions for any $(x,y) \in \mathcal{S} \times \mathcal{S}$.
\end{enumerate}
\end{te}
\begin{proof}
\begin{enumerate}
\item[ 1) $\to$ 2)] By using Theorem \ref{tetimech} it is sufficient to prove that any completely monotone function which is also the survival function of a non-negative r.v.'s (with a Lebesgue density diverging at zero) can be considered as the moment generating function of an inverse subordinator $L^f$ for some $f$. In other words we prove that if $t \mapsto \overline{F}_{x,y}(t)$ are completely monotone functions for any $(x,y)$, then it must be true that
\begin{align}
\overline{F}_{x,y}(t) \, = \, \mathds{E}^ze^{-\theta(x) L^{(x,y)}(t)}
\end{align}
for some inverse $L^{(x,y)}(t)$ and hence by Theorem \ref{tetimech} we have $X(t) = M (L^\Pi(t))$ where $\Pi$ is defined in \eqref{lapl}.
After this we prove that the representation in \eqref{reprf} is true.
Observe that the dependence of $f$ on $(x,y)$ is unnecessary in this proof: it is indeed sufficient to prove that every completely monotone function can be written as the moment generating function of an inverse subordinator $L^f$ for some Bernstein function $f$. Then letting $(x,y)$ vary is the same thing as considering a different function $f$. Hence in what follows we omit the dependence of $\overline F_{x,y}$ from $(x,y)$. \br{In the same spirit} we also fix the parameter $\theta(x) = \theta$.

Since $t \mapsto \overline{F}(t)$ is completely monotone then
\begin{align}
\overline{F}(t) = \int_0^\infty e^{-ts} K(ds)
\label{cmg}
\end{align}
for some measure $K(\cdot)$ with $K(0, \infty) =1$. Hence
\begin{align}
\widetilde{\overline{F}}(\lambda) = \int_0^\infty \frac{1}{s+\lambda} K(ds)
\label{gsti}
\end{align}
which is a Stieltjes function if
\begin{align}
\int_0^\infty \frac{1}{s+1} K(ds) < \infty.
\label{convst}
\end{align}
But \eqref{convst} is true since $K[0, \infty) = 1$.
Hence by Lemma \ref{te22} we know that there exists $f(\lambda)$ complete Bernstein function such that
\begin{align}
\widetilde{\overline{F}}(\lambda) = \frac{f(\lambda)}{\lambda} \frac{1}{\br{\theta}+f(\lambda)}.
\label{423}
\end{align}
Use \cite[Corollary 3.5]{meertri} to say that
\begin{align}
\int_0^\infty e^{-\lambda t} \mathds{E}^ze^{-\br{\theta}L^f(t)} dt \, = \, \frac{f(\lambda)}{\lambda} \frac{1}{\br{\theta}+f(\lambda)}
\end{align}
and hence since $t \mapsto \mathds{E}e^{-\br{\theta} L^f(t)}$ is completely monotone by Lemma \ref{tecommon} and obviously continuous we have that $\overline{F}(t) = \mathds{E}^xe^{-\br{\theta}L^f(t)}$ for any $t \geq 0$.

Now we prove that the representation for $f$ is true.
Rearrange \eqref{423} and use Lemma \ref{te22} to say that
\begin{align}
f(\lambda) \, = \,\br{\theta} \frac{\lambda \widetilde{\overline{F}}(\lambda)}{1-\lambda \widetilde{\overline{F}}(\lambda)}
\label{431}
\end{align}
is a complete Bernstein function.
Hence we have
\begin{align}
f(\lambda) = \br{\theta}\l a+b\lambda + \int_0^\infty \l 1-e^{-\lambda s} \r \nu(ds)\r,
\end{align}
where \br{the L\'evy measure $\nu(ds)$ has a completely monotone density $s \mapsto \nu(s)$ since $f$ is complete. Now we prove that $a=0$, $b=0$ and $\nu(0, \infty) = \infty$. By using \eqref{gsti} it is easy to verify that $\lim_{\lambda \to \infty}f(\lambda) = \infty$: indeed note that}
\begin{align}
\lim_{\lambda \to \infty} \lambda \widetilde{\overline{F}}(\lambda)\, = \, & \lim_{\lambda \to \infty} \int_0^\infty \frac{\lambda}{\lambda +s} K(ds) \notag \\
= \, & \lim_{\lambda \to \infty} \int_0^\infty \l 1-\frac{s}{\lambda + s}\r K(ds)
\label{335}
\end{align}
and since $\lambda \mapsto \l 1-s (\lambda +s)^{-1} \r$ is \br{bounded} and monotone we can move the limit inside the integral to say that $\lim_{\lambda \to \infty} \lambda \widetilde{\overline{F}}(\lambda) = 1$. We can conclude now that $\lim_{\lambda \to \infty } f(\lambda) = \infty$. Therefore it must be true, by \cite[Corollary 3.7, Item (v)]{librobern}, that $\nu(0, \infty) = \infty$ or $b>0$ (or both). 
Now we compute $b$. Note that by \cite[p. 16, item (iv)]{librobern} we know that
\begin{align}
b \, = \, \lim_{\lambda \to \infty} \lambda^{-1} f(\lambda)
\end{align}
and by using \eqref{431} we have that
\begin{align}
b \, = \, & \lim_{\lambda \to \infty} \frac{\widetilde{\overline{F}}(\lambda)}{1-\lambda \widetilde{\overline{F}}(\lambda)} \notag \\
= \, & \lim_{\lambda \to \infty} \frac{\int_0^\infty (s+\lambda)^{-1} K(ds)}{1-\lambda \int_0^\infty \l s+\lambda\r^{-1}  K(ds) }.
\label{limitforb}
\end{align}
However since $K(0, \infty) =1$  we note that the denominator is given by
\begin{align}
1-\lambda  \widetilde{\overline{F}}(\lambda) \, = \, & 1- \lambda \int_0^\infty \l s+\lambda\r^{-1}  K(ds) \notag  \\
= \, & \int_0^\infty \l 1- \frac{\lambda}{s+\lambda} \r K(ds) \notag \\
= \, & \int_0^\infty \frac{1}{s+\lambda} sK(ds).
\label{dens}
\end{align}
Let $k(ds): = sK(ds)$, the limit \eqref{limitforb} is
\begin{align}
b \, = \, &\lim_{\lambda \to \infty} \frac{\int_0^\infty (\lambda + s)^{-1} s^{-1}k(ds) }{\int_0^\infty (\lambda + s)^{-1} k(ds)} \notag \\
= \, & \lim_{\lambda \to \infty} \frac{\int_0^\infty \lambda^{-1} \l 1 +\lambda^{-1} s \r^{-1}s^{-1}k(ds)}{ \int_0^\infty \lambda^{-1} \l 1+\lambda^{-1} s \r^{-1} k(ds)} \notag \\
= \, & \lim_{\lambda \to \infty} \frac{\int_0^\infty  \l 1 +\lambda^{-1} s \r^{-1}s^{-1}k(ds)}{ \int_0^\infty  \l 1+\lambda^{-1} s \r^{-1} k(ds)}.
\label{227}
\end{align}
The function $\lambda \mapsto \l 1 +\lambda^{-1} s \r^{-1}$ is, for any fixed $s>0$, \br{bounded} and monotone. Hence in \eqref{227} we can move the limit inside the integral to get
\begin{align}
b \, = \,  \frac{\int_0^\infty s^{-1}k(ds)}{ \int_0^\infty  k(ds)} \, = \, \frac{1}{k(0, \infty)}.
\label{b1suk}
\end{align}
Hence we have by the assumptions that $b=0$ since $k(0, \infty) = \int_0^\infty sK(ds) = \infty$.
The fact that $a=0$ is true since $a = f(0)$ and we can use again the representation \eqref{335} to verify. This complete the proof.

\item[2) $\to$ 1)] Now we prove the converse statement.  If $X(t)$ and $M \l L^\Pi(t) \r$ are the same process then it must be true by Theorem \ref{tetimech} that $\overline{F}_{x,y}(t) = \mathds{E}^ze^{-\theta(x) L^{(x,y)}(t)}$ for any $(x,y)$ and that $L^{(x,y)}$ are inverses subordinators with Laplace exponents $f(\lambda, x,y)$ given in \eqref{320}. Further since $f(\lambda, x, y)$ are complete Bernstein functions then we have by Lemma \ref{tecommon} that $t \mapsto \overline{F}_{x,y}(t)$ is completely monotone with respect to a measure $K$. The fact that $K[0, \infty)=1$ is obvious since $L^{(x,y)} = 0$ a.s. The fact that $\int_0^\infty sK(ds) = \infty$ can be verified by repeating the computation leading to \eqref{b1suk}. 
\end{enumerate}
\end{proof}
\begin{os} \normalfont
Note that the condition assumed in Theorem \ref{tecm}
\begin{align}
\int_0^\infty sK_{x,y}(ds) = \infty
\end{align}
together with the complete monotonicity of $t \mapsto \overline{F}_{x,y}(t)$ \br{implies that the the corresponding r.v.'s have densities $\mathpzc{f}_{x,y}(\cdot)$ which are singular at zero.}
\end{os}
\begin{os} \normalfont
Suppose that the measures $K_{x,y} (\cdot)$ have the form the so-called Lamperti distribution \cite{lamperti}, for $\alpha(x,y) \in (0,1)$,
\begin{align}
K_{x,y}(ds) \, = \, \frac{\sin \pi \alpha(x,y)}{\pi} \frac{s^{\alpha(x,y) -1}}{s^{2\alpha(x,y)}+2y^{\alpha(x,y)} \cos \pi \alpha(x,y) +1}ds.
\end{align}
We have \cite{gross}
\begin{align}
\int_0^\infty e^{-s  t} \frac{\sin \pi \alpha(x,y)}{\pi} \frac{s^{\alpha(x,y) -1}}{s^{2\alpha(x,y)}+2y^\alpha(x,y) \cos \pi \alpha(x,y) +1}ds \, = \, E_{\alpha(x,y)}\l - t^{\alpha(x,y)} \r
\end{align}
where
\begin{align}
E_{\alpha(x,y)}\l x \r \, : = \, \sum_{k=0}^\infty \frac{x^k}{\Gamma (\alpha(x,y) k+1)}
\end{align}
is the Mittag-Leffler function. Then since (e.g. \cite[eq. (3.4)]{meerbounded})
\begin{align}
\int_0^\infty e^{-\lambda t} E_{\alpha(x,y)}\l - t^{\alpha(x,y)} \r dt \, = \, \frac{\lambda^{\alpha(x,y)-1}}{1+\lambda^{\alpha(x,y)}}
\end{align}
the representation \eqref{reprf} yields
\begin{align}
\Pi (\lambda, t) \, = \, \sum_{k=0}^{n^J(t)-1} \l \tau_{k+1}-\tau_k \r \lambda^{\alpha(x_k, x_{k+1})} + (t-\tau_{n^J(t)}) \lambda^{\alpha(x_{n^J(t)}, x_{n^J(t)+1})}.
\end{align}
Hence the \br{associated subordinators representing the increments of $\sigma^\Pi$ are $\alpha(x,y)$-stable subordinators conditionally on $X_0=x, \cdots, X_{n+1}=x_{n+1}, \mathcal{T}_1=\tau_1, \cdots, \mathcal{T}_n=\tau_n$}.
\end{os}

\section{The Kolmogorov's equations}
When a semi-Markov process $X$ is given by the time change of a Markov process $M$ by means of an independent inverse subordinator the corresponding Kolmogorov's backward equation has been written down in some different ways. In particular, \cite{meertri} showed that the scale limit of a (suitably defined) continuous-time random walk is of the form $M \l L^f(t) \r$, where $M$ is a L\'evy process generated by $A$, then its pdf $p(x, t)$ solves a governing equation
\begin{align}
\mathbb{C}_f \l \partial_t \r p(x, t) \, = \, Ap(x, t)
\end{align}
where
\begin{align}
\mathbb{C}_f \l \partial_t \r u(t) \, : = \, \mathcal{L}^{-1} \left[ f(\lambda) \widetilde u(\lambda) - \lambda^{-1}f(\lambda) u(0) \right] (t).
\label{meerinv}
\end{align}
In \cite{toaldopota} the author introduced the operator
\begin{align}
\mathcal{D}^f_t q(t) : = \frac{d}{dt} \int_0^t q(s) \, \bar{\nu}(t-s) ds
\end{align}
where
\begin{align}
\bar{\nu}(t) : = a+ \nu(t, \infty)
\end{align}
which can be regularized by subctracting an initial condition (see also \cite[Remark 4.8]{meertri}) as
\begin{align}
\mathfrak{D}_t^f q(t)\, := \, & \mathcal{D}^f_tq(t)- \bar{\nu}(t) q(0).
\label{caputotype}
\end{align}
The operators \eqref{caputotype} and \eqref{meerinv} turn out to be the same operator at least for exponentially bounded continuously differentiable functions (see the discussion following \cite[formula (2.18)]{meertoa}).
Suppose $P_t$ is a semigroup of linear operators on a Banach space $\mathfrak{B}$ corresponding to a Markov process $M$, i.e.,
\begin{align}
\l P_tu \r (x) \, := \, \mathds{E}^xu(X(t))
\end{align}
and then define
\begin{align}
(\Pi_t u)(x)\, = \, \mathds{E}^x u \l M \l L^f(t) \r \r
\label{timenaif}
\end{align} 
where $L^f$ is an independent inverse subordinator.
Using Bochner integrals \eqref{timenaif} is equivalent to
\begin{align}
\Pi_tu \, = \, \int_0^\infty P_su \; l(ds,t)
\end{align}
where $l(ds, t) := P^x \l L^f(t) \in ds \r$. In this case we can use \cite[Theorem 5.1]{toaldopota} to say that $t \mapsto P_tu$ solves
\begin{align}
\mathfrak{D}_t^f q(t) \, = \, Aq(t) \qquad q(0) = u \in \text{Dom}(A).
\label{eqold}
\end{align}
When $f^\star(\lambda):= \lambda /f(\lambda)$ is again a Bernstein function it is proved in \cite{meertoa} (under the additional assumption that $A$ is self-adjoint on an Hilbert space) that equation \eqref{eqold} can be rewritten, for $t>0$, in the evolutionary form
\begin{align}
\frac{d}{dt}q(t) \, = \, \mathcal{D}_t^{f^\star} Aq(t), \qquad q(0) = u,
\label{k2}
\end{align}
where $\mathcal{D}_t^{f^\star}$ is related to $f^\star$ in the same way $\mathcal{D}^f$ is related to $f$. Hence when $M(L^f(t))$  is a semi-Markov process then \eqref{eqold} and \eqref{k2} play the same role of the classical backward Kolmogorov's equation of Markov processes. In \cite{magda} an evolutionary form analogous to \eqref{k2} is proposed. In what follows we provide two integro-differential forms of the Kolmogorov's backward equation for the general class of semi-Markov processes considered above.
\subsection{The backward equation}
Let $C_b (\mathcal{S})$ be the space of bounded functions on $\mathcal{S}$ equipped with the sup-norm $\left\| \cdot \right\|_\infty$ and note that these functions are also continuous since $\mathcal{S}$ is equipped with the discrete topology. We define the Markov semigroup $P_t$ acting on functions $u \in C_b (\mathcal{S})$ associated with the Markov process $M(t)$ as
\begin{align}
(P_tu)(x) : = \mathds{E}^x u(M(t)).
\end{align}
The generator of $P_t$ is the operator $G$ given by
\begin{align}
(G u)(\cdot) \, := \, \theta(\cdot)\int_{\mathcal{S}} \l u(y)-u(\cdot) \r h(\cdot,dy)
\label{defG}
\end{align}
which is a bounded operator with respect to $\left\| \cdot \right\|_\infty$ and hence $P_t$ is strongly continuous on $C_b \l \mathcal{S} \r$ (e.g. \cite[Theorem 1.2, p. 2]{pazy}).
Hence the mapping $t \mapsto q(t) : = P_tu$ is $C^1([0, \infty), C_b(\mathcal{S}))$ and is the unique classical solution to the Cauchy problem (e.g. \cite[Proposition 6.2]{engelnagel})
\begin{align}
\frac{d}{dt}q(t) = Gq(t), \qquad q(0) = u,
\label{mbackcon}
\end{align}
for any $u \in C_b (\mathcal{S})$. In the forthcoming theorems we propose two equivalent very general forms for the Kolmogorov's equations of stepped semi-Markov processes. Hence define the operator $\Pi_t^s$, on the space $C_b(\mathcal{S})$ of continuous and bounded functions on $\mathcal{S}$, as
\begin{align}
(\Pi_t^s u) (x):= \mathds{E}^x \left[ u(X(t)) \mid \gamma^X(0) = s \right] \, = \, \mathds{E}^y \left[ u(X(t+\tau)) \mid X(\tau) =x , \gamma^X(\tau) =s \right]
\end{align}
for any $y \in \mathcal{S}$ and $t, \tau \geq 0$, $s\leq t$. Therefore, in order to study integro-differential equations we need some properties of the operator $\Pi_t^s$ and of the mapping $t \mapsto \Pi_t^su$. In particular we will write the Kolmogorov's equation of the process in two integro-differential forms, and hence we will then consider the case $s=0$, i.e., the mapping
\begin{align}
t \mapsto q(t) := \Pi_t u := \Pi_t^0u.
\end{align}
We have the following two results.
\begin{prop}
Let $u \in C_b\l \mathcal{S}\r$. Then $\Pi_t^su \in C_b \l \mathcal{S} \r$ for any $t \geq 0$.
\end{prop}
\begin{proof}
Consider the semigroup of operators associated with the strict Markov process $\l X(t), \gamma^X(t) \r$, i.e., the operators
\begin{align}
(T_t h)(x,s) : = \mathds{E}^{x} \left[ h(X(t), \gamma^X(t)) \mid \gamma^X (0) = s\right]
\label{semcoppia}
\end{align}
on the space $C_b(\mathcal{S} \times \mathbb{R}^+)$. Continuity of functions $h(x,s)$ on $\mathcal{S} \times \mathbb{R}^+$ simply means continuity in $s$ since $\mathcal{S}$ is equipped with the discrete topology. By \cite[Theorem 1, p. 239]{gihman} we know that $T_t$ maps $C_b \l \mathcal{S} \times \mathbb{R}^+ \r$ into itself. Now set $h(x,y) = u(x)$ for any $y \geq 0$ and note that
\begin{align}
(T_t h)(x,s) \, = \,& \int_{\mathcal{S}\times \mathbb{R}^+} u(y) P^{x} \l X(t) \in dy, \gamma^X (t) \in dw \mid \gamma^X (0) =s \r \notag \\
= \, & \int_{\mathcal{S}} u(y) P^{x} \l X(t) \in dy \mid \gamma^X (0) =s \r \notag \\
= \, &(\Pi_t^s u)(x).
\label{4166}
\end{align}
Since $T_th \in C_b \l \mathcal{S} \times \mathbb{R}^+ \r$ we have by \eqref{4166} that $\Pi_t^s u \in C_b \l \mathcal{S} \r $.
\end{proof}
\begin{prop}
\label{tediff}
Let $u \in C_b(\mathcal{S})$.
If $u$ is such that
\begin{align}
\lim_{t \to 0+} \frac{1}{t}\l \Pi_t^s u-u \r
\end{align}
exists as strong limit in $C_b \l \mathcal{S} \times [0, \infty) \r$ and if $t \mapsto F_{x,y}(t)$ is continuous for any $(x,y) \in \mathcal{S} \times \mathcal{S}$, then the mapping $t \mapsto \Pi_t^su$ is $C^1 \l [0, \infty), C_b (\mathcal{S}) \r$.
\end{prop}
\begin{proof}
Consider again the semigroup of operators associated with the strict Markov process $\l X(t), \gamma^X(t) \r$, i.e., the operators defined in \eqref{semcoppia}, on the space $C_b(\mathcal{S} \times \mathbb{R}^+)$. Recall again that continuity of functions $h(x,s)$ on $\mathcal{S} \times \mathbb{R}^+$ simply means continuity in $s$ since $\mathcal{S}$ is equipped with the discrete topology. The family $\l T_t \r_{t \geq 0}$ forms a (strongly continuous) $C_b(\mathcal{S}\times \mathbb{R}^+)$-Feller semigroup (see \cite[formula (33) p. 238 and Theorem 1 p. 239]{gihman}). Denote by $G^T$ the generator of $T_t$. Now set $h(x,y) = u(x)$ for any $y \in \mathbb{R}^+$. We have by \eqref{4166}
\begin{align}
(T_t h)(x,s) \, = \, (\Pi_t^su )(x).
\end{align}
Since $T_t$ is a strongly continuous Feller semigroup we know that the mapping $t \mapsto T_th$ is $C^1$ for any $h \in C_b(\mathcal{S} \times \mathbb{R}^+)$ in the domain of $G^T$ and hence it is sufficient to prove that $h(x,s):=u(x)$ is $C_b (\mathcal{S} \times \mathbb{R}^+)$ and is in the domain of $G^T$. The fact that $h(x,s) = u(x) \in C_b(\mathcal{S} \times \mathbb{R}^+)$ is easy to see since we assume that $u(x) \in C_b(\mathcal{S})$. Now we prove that $h(x,s) = u(x) \in \text{Dom}(G^T)$. This can be shown by considering, as $t \to 0$,
\begin{align}
\left\| \frac{1}{t} \l T_th-h \r  \right\| \, = \, \left\| \frac{1}{t} \l \Pi_t^s u-u\r \right\|
\label{limit}
\end{align}
and the limit on the right-hand side of \eqref{limit} exists by the hypotheses for any $s$.
\end{proof}
The following Theorem provides the first version of the Kolmogorov's equation for semi-Markov processes in the form of an integro-differential Volterra equation. We remark that from now on it will be always true that the c.d.f. $F_{x,y}$ is independent on $y$.
\begin{te}
\label{teclassiceq}
Let $u$ be as in Proposition \ref{tediff}. Assume that for each $x \in \mathcal{S}$ it is true that $\overline{F}_x(t) = \mathds{E}^y e^{-\theta(x)L^x(t)} $ for some $f(\lambda, x)$, with $\nu((0, \infty),x) = \infty$ and $s \mapsto \bar{\nu}(s,x)$ continuous for any $x$.
The mapping $t \mapsto q(t):=\Pi_tu$ satisfies the initial value problem
\begin{align}
\mathfrak{D}_t^{\bm{\cdot}} q(t) \, = \, G q(t), \qquad q(0) = u,
\label{classiceqgen}
\end{align}
where
\begin{align}
\l \mathfrak{D}_t q(t) \r \, (\cdot) \, : = \, \frac{d}{dt} \int_0^t q(s, \cdot) \, \bar{\nu}(t-s, \cdot) \, ds \,- \, \bar{\nu}(t, \cdot)q(0, \cdot)
\end{align}
with $q(s, \cdot):= (\Pi_s u)(\cdot)$ and $G$ is defined in \eqref{defG}.
\end{te}
\begin{proof}
Under our assumptions on $L^x (t)$ (section \ref{sectimech}) we have that the functions $t \mapsto \mathds{E}e^{-\theta(x)L^{x}(t)}$ are continuous functions (see the proof of \cite[Theorem 3.1]{meertri}). Hence the conditions of Proposition \ref{tediff} are satisfied and $t \mapsto q(t)$ is continuously differentiable. This fact together with the continuity of $s \mapsto \bar{\nu}(s)$ ensures that the operator $t \mapsto  \mathfrak{D}_t^{\bm{\cdot}}q(t)$ is well defined.
By \cite[eq. (1.52), p. 20]{koro} we know that the mapping $q(t)$ satisfies
\begin{align}
(\Pi_t u)(x) = \overline{F}_x (t) u(x) + \int_0^t \int_{\mathcal{S}} (\Pi_{t-s}u )(y) \, h(x,dy) \, \mathpzc{f}_x(s)ds.
\label{koroswi}
\end{align}
Let $(\mathpzc{R}_\lambda u)(x) := \int_0^\infty e^{-\lambda t} (\Pi_t u)(x) dt$. By \eqref{319}, \eqref{3200} and the convolution theorem for Laplace transform we can take Laplace transform in \eqref{koroswi} to write
\begin{align}
(\mathpzc{R}_\lambda u)(x) = \frac{f(\lambda,x)}{\lambda} \frac{1}{\theta(x)+f(\lambda, x)}u(x) \, + \, \frac{\theta(x)}{\theta(x)+f(\lambda, x)} \int_\mathcal{S} (\mathpzc{R}_\lambda u) (y) h(x,dy).
\label{laplrlambda}
\end{align}
By rearranging \eqref{laplrlambda} we get
\begin{align}
f(\lambda, x) (\mathpzc{R}_\lambda u) (x) - \lambda^{-1}f(\lambda, x) u(x) \, = \, G (\mathpzc{R}_\lambda u)(x).
\label{rearranged}
\end{align}
It is easy to check that \eqref{rearranged} is the Laplace transform of \eqref{classiceqgen} by using again the convolution theorem and the fact that
\begin{align}
\int_0^\infty e^{-\lambda t} \bar{\nu}(t,x) \, dt \, = \, \lambda^{-1} f(\lambda, x).
\end{align}
\end{proof}
\begin{os} \normalfont
\label{remfrac}
We remark that the operator $\mathfrak{D}_t^{\bm{\cdot}}g(t)$ acting on functions $g:[0, \infty) \mapsto C_b(\mathcal{S})$ can be interpreted as a generalized ``variable-order" fractional derivative. Suppose that for any $x$
\begin{align}
\overline{F}_x(t) \, = \, E_{\alpha(x)}(-\theta(x) t^{\alpha(x)}),
\label{mittsf}
\end{align}
where 
\begin{align}
E_\alpha(z) : = \sum_{k=0}^\infty \frac{z^k}{\Gamma(\alpha k+1)}
\end{align}
is the Mittag-Leffler function. We obtain from \eqref{mittsf} the Laplace transform
\begin{align}
\widetilde{\overline{F}}_x(\lambda) \, = \, \lambda^{\alpha(x)-1} \frac{1}{\theta(x)+\lambda^{\alpha(x)}}.
\end{align}
Furthermore since \eqref{mittsf} is completely monotone and with first derivative diverging at zero (e.g. \cite[Section 3.1]{mainardimittag}) we can apply Theorem \ref{tecm} to say that $\overline{F}_x(t) = \mathds{E}^ze^{-\theta(x)L^{(x)}(t)}$ where $L^{(x)}(t)$ are a family of inverses of $\alpha(x)$-stable subordinators, i.e., the representation \eqref{reprf} yields to
\begin{align}
f(\lambda,x) \, = \, \lambda^{\alpha(x)}.
\end{align}
This implies $\bar{\nu}(s, x) = s^{-\alpha(x)}/\Gamma(1-\alpha(x))$ and hence
\begin{align}
\l \mathfrak{D}_t^{\bm{\cdot}} g(t) \r (\cdot) \, = \,& \frac{1}{\Gamma (1-\alpha (\cdot))} \frac{d}{dt} \int_0^t g(s,\cdot) \, (t-s)^{-\alpha(\cdot)}ds - \frac{t^{-\alpha(\cdot)}}{\Gamma(1-\alpha(\cdot))}u(\cdot)\label{fracvarord} \\
 =: \, & \l \frac{d^{\alpha(\cdot)}}{dt^{\alpha(\cdot)}}g(t) \r (\cdot)\notag.
\end{align}
In formula \eqref{fracvarord} we recognize the fractional derivative in the sense of Caputo-Dzerbayshan (e.g. \cite[eq. (6.1.39)]{kilbas}).
\end{os}
We show now that also eq \eqref{k2} can be extended to our framework. Since we do not assume in this paper that our subordinators are special (as it was assumed in \cite{meertoa}) we write here an equation on the same line of \eqref{k2} valid for general subordinators.
We consider here the operator $\mathcal{D}_t^\star g(t)$, acting on functions $g:[0, \infty) \mapsto C_b(\mathcal{S})$, given by
\begin{align}
\l \mathcal{D}_t^\star g(t) \r (\cdot) \, : = \, \frac{d}{dt} \int_0^t g(\cdot,s) \, u^f(t-s,\cdot) \, ds
\label{cappotmeas}
\end{align}
where $u^f(s,\cdot)$ is the potential density of a subordinator with Laplace exponent $f(\lambda, \cdot)$, $\cdot \in \mathcal{S}$. Of course it is not necessarily true that such a density exists. For example if $s \mapsto \bar{\nu}(s,\cdot)$ is absolutely continuous and $\nu((0, \infty), \cdot) = \infty$ then the density exists since the distribution of $\sigma^{(\cdot)} (t)$ has a density (by \cite[Theorem 27.7]{satolevy}) which we denote by $s \mapsto \mu(s, t)$ and hence we have
\begin{align}
s \mapsto u^f(s,\cdot)ds \, = \, \mathds{E}^x \int_0^\infty \mathds{1}_{\left[ \sigma^{(\cdot)}(t) \in ds \right]} dt \, = \, ds\int_0^\infty \mu (s, t) dt.
\end{align}
Using \eqref{cappotmeas} we have the following evolutionary form for the equation \eqref{classiceqgen}.
\begin{te}
\label{teeqevol}
Let $u$ be as in Proposition \ref{tediff}. Assume that for any $x \in \mathcal{S}$ it is true that $\overline{F}_{x,y}(t) =\mathds{E}^{y}e^{-\theta(x)L^{(x)}(t)}$ for some $f(\lambda, x)$ with $\nu((0, \infty),x) = \infty$ and assume that there exists a potential density $s \mapsto u^f(s,x)$ continuous.
Then $t \mapsto q(t)$ satisfies the evolutionary Cauchy problem
\begin{align}
\frac{d}{dt} q(t) \, = \,\mathcal{D}_t^\star  \;G   q(t), \qquad q(0) = u.
\label{evoleq}
\end{align}
\end{te}
\begin{proof}
Of course the operator $\mathcal{D}_t^\star q(t)$ is well defined since $t \mapsto q(t)$ is continuously differentiable and $s \mapsto u^f(s, \cdot)$ is continuous. Hence to prove this theorem it is sufficient to rearrange \eqref{rearranged} multiplying both sides by $\lambda /f(\lambda, x)$ to write
\begin{align}
\lambda (\mathpzc{R}_\lambda u)(x) - u (x) \, = \, \frac{\lambda}{f(\lambda, x)} G (\mathpzc{R}_\lambda u) (x).
\label{427}
\end{align}
By the convolution theorem for the Laplace transform and since the Laplace transform of the potential density $u^f(s, \cdot)$ is
\begin{align}
\int_0^\infty e^{-\lambda s} u^f(s, \cdot) \, ds \, = \, \frac{1}{f(\lambda, \cdot)},
\end{align}
it is easy to see that \eqref{427} is the Laplace transform of \eqref{evoleq}.
\end{proof}
\begin{os} \normalfont
Suppose we are in the situation of Remark \ref{remfrac}. Then the potential densities corresponding to the $\alpha(x)$-stable subordinators are
\begin{align}
u^f(s, \cdot) \, = \, \frac{s^{\alpha(\cdot)-1}}{\Gamma(\alpha(\cdot))}.
\end{align}
Hence the operator $\mathcal{D}^\star $ becomes
\begin{align}
\l \mathcal{D}_t^\star q(t) \r (\cdot) \, = \, \frac{d^{1-\alpha(\cdot)}}{dt^{1-\alpha(\cdot)}}q(t, \cdot)
\end{align}
but the fractional derivative must be meant now in the Riemann-Liouville sense (e.g. \cite[p. 69]{kilbas}).
\end{os}
\begin{ex}[Time-changed Poisson process]
\label{expoieq}
Suppose that the r.v.'s $J_i$ have distribution $P^x (J_i > t \mid X_n = z, X_{n+1} = y) = e^{-\theta t}$ and that $X_n$ is a discrete Markov chain in $\mathbb{R}^d$ with transition probabilities $h(x,dy) = \delta_{x+l}(dy)$ where $\delta_z(\cdot)$ is the Dirac point mass at $z \in \mathbb{R}^d$. Then define
\begin{align}
N^l(t) \, = \, X_n, \qquad \sum_{i=0}^n J_i \leq t < \sum_{i=0}^{n+1} J_i.
\end{align}
Hence $N^l(t)$ is a Poisson process in $\mathbb{R}^d$ with jump height $l \in \mathbb{R}^d$.
Now let $\mathpzc{J}_i$ be independent r.v.'s with distribution
\begin{align}
P^x \l \mathpzc{J}_n > t \mid X_n=z, X_{n+1} = y \r \, = \,  \overline{F}_z(t)
\end{align}
and define
\begin{align}
\mathpzc{N}^l(t) \, = \, X_n, \qquad \sum_{i=0}^n \mathpzc{J}_i \leq t < \sum_{i=0}^{n+1} \mathpzc{J}_i.
\end{align}
Hence $\mathpzc{N}^l (t)$ is a semi-Markov process in the sense of Gihman and Skorohod defined as in Section \ref{defsemi}.
Now assume that $t \mapsto \overline{F}_y(t)$ is completely monotone. Hence by Theorem \ref{tecm} we know that there exists a process $L^\Pi$ defined as in section \ref{sectimech} by setting for any $x \in \mathbb{R}^d$,
\begin{align}
f(\lambda, x) = \frac{\lambda\widetilde{\overline{F}}_x (\lambda)}{1-\lambda \widetilde {\overline{F}}_x(\lambda)},
\label{easyto}
\end{align}
such that $\mathpzc{N}^l(t)$ and $N^l \l L^\Pi (t) \r $ are the same process. Furthermore since we have by Theorem \ref{tecm} that $\lambda \mapsto f(\lambda, x)$ is, for any $x \in \mathbb{R}^d$, a complete Bernstein function then we know that the L\'evy measures $\nu(\cdot,x)$ have a completely monotone density and hence also the tails $s \mapsto \bar{\nu}_x(s)$ are completely monotone. Furthermore also the potential densities $s \mapsto u^f(s, x)$ exist and are completely monotone \cite[Remark 10.6]{librobern}, for any $x \in \mathbb{R}^d$. Hence all the continuity properties are satisfied and we can apply Theorems \ref{teclassiceq} and \ref{teeqevol} to say that the mapping $t \mapsto q(t,x): = \mathds{E}^xu(\mathpzc{N}^l(t))$, for $u \in C_b \l \mathbb{R}^d \r$ satisfying the condition of Proposition \ref{tediff}, solves
\begin{align}
\frac{d}{dt} \int_0^t q(s,x) \, \bar{\nu}(t-s,x)ds-\bar{\nu}(t,x)q(0,x) \, = \, \frac{\theta}{|l|} \l q(t,x+l)-q(t,x)\r,
\end{align}
under $q(0,x) = u(x)$, as well as
\begin{align}
\frac{d}{dt} q(t,x) \, = \,\frac{\theta}{|l|} \frac{d}{dt} \int_0^t \l q(s,x+l)-q(s,x) \r \, u^f(t-s,x) \, ds,
\end{align}
under $q(0,x) = u(x)$.
See also Garra et al. \cite{garra} for a similar equation related to a state-dependent Poisson model.
\end{ex}
\section{Geneneral semi-Markov processes as limit of stepped semi-Markov processes}
In the present section we study weak limits of  semi-Markov processes defined above. Denoting by $M^c$  a stepped Markov process depending on a parameter $c$, and by $X^c(t)=M^c (L(t))$ the semi-Markov process obtained by the time-change with a dependent random time $L^\Pi$, we study the limit as $c\to 0$ of the single time distribution of $X^c$. Formally, assuming that the generator $G^c$ of the semigroup $P_t^c$ associated with $M^c$ converges to a generator $G$ on a suitable function space, we are able to study the convergence of $\Pi_t^cu$ for $u \in C_b (\mathcal{S})$.

Then by assigning to the parameter $c$ different meanings we can reexamine some classical results on the convergence of Markov processes, in our semi-Markov framework.
It is a known fact that any L\'evy process is the limit as $c\to 0$ of a suitable sequence of compound Poisson processes, say $M^c$, where $c$ represents the lower bound for the jump sizes; an analogous result holds for L\'evy-type processes in the sense of \cite[p. 366]{kolokoltsov}, i.e., jump processes whose L\'evy measures not only depend on the jump size, but also on the current position. Moreover, we remind that some stepped Markov processes are continuous time random walks \br{in $\mathbb{R}^d$ jumping on a lattice of size $c$}. By letting $c\to 0$ (and simultaneously scaling the intensity of jumps) \br{one obtains processes whose paths are no more stepped functions}. Following this procedure, the Brownian motion is known to be obtained as a limit of symmetric random walks.
Hence we study here the semi-Markov analogue of these facts, where a sequence of semi-Markov processes $X^c$ satisfies the assumptions in Theorem \ref{tetimech} or \ref{tecm} and hence $M^c$ is time-changed by a dependent process $L^\Pi$. 

An interesting case is the one leading to a diffusive limit. Very recently, indeed, fractional diffusion equations have been derived, exhibiting a  variable order fractional derivative (see \cite{baestra, checgore, fedofalco}). These are models of anomalous diffusion in inhomogeneous media. In the heuristic discussion in \cite{checgore}, it is supposed that each lattice site has a trapping effect which leads to a sub-diffusive dynamics; in particular, the distribution of the sojourn time is here assumed to depend on the site itself, and this leads to a state dependent anomalous exponent. The study of this matter seems to be still at an early stage. Our results specialized to this situation provide a probabilistic derivation of the variable order diffusion equation. 

The results below also extends the analysis given in \cite{kurtz}, where asymptotic results on semi-Markov processes where discussed by only considering state dependent holding times with finite expectation. Our processes $M^c$ can have holding times having infinite mean and this turns out to be the most interesting situation.
In particular section \ref{seclimit} is devoted to the case in which the waiting times $J_i$ have infinite mean when this is obtained by requiring that the Bernstein functions $f(\lambda, x)$ defining the exponent $\Pi$ are regularly varying at zero. By giving to $c$ the meaning of a scale parameter the result in Theorem \ref{telimit} below provides a further probabilistic derivation of the (variable order) $\alpha(x)$ diffusion equation (together with other related results). This is done by means of a local scaling of the lattice size, in which the scaling factor is given by the regularly varying function $f(c, x)$ depending on the current position $x$: passing to the limit, this gives rise to the fractional operator of state dependent order.
\subsection{Limit of stepped semi-Markov processes}
\label{seclimfacile}
Here is the first general statement. The following theorem assumes the convergence of the generator $G^c$ to a generator $G^0$ and study the convergence of the corresponsing semi-Markov process $X^c(t)$ constructed as in Section \ref{sectimech}.
\begin{te}
\label{telimfacile}
Let $M^c$ be a family of Markov processes associated with the semigroups $P_t^c$ on the space $C_b \l \mathcal{S} \r$ having generators $G^c$ which can be written in the form $(G^cu)(\cdot):=\theta^c(\cdot) \int_S \l u(y)-u(\cdot) \r h^c(\cdot,dy)$. Let $G^0$ be a (dissipative) generator of a semigroup on the Banach space $C_0(\mathcal{S})$, where $C_0(\mathcal{S})$ the space of continuous functions vanishing at infinity, and assume that $G^cu \to G^0u$  for $u \in \text{Dom}(G^0)$. Let $X^c(t)$ be a family of stepped semi-Markov processes defined as in Section \ref{defsemi} each one of which satisfies the assumptions of Theorem \ref{tetimech} with respect to the Markov process $M^c$ and hence denote
\begin{align}
\l \Pi_t^c u \r (x) : = \mathds{E}^x u(X^c(t)) \, = \, \mathds{E}^x u \l M^c \l L^\Pi (t) \r \r.
\end{align}
Let $\mathpzc{R}^c_\lambda u:= \int_0^\infty e^{-\lambda t} \Pi_t^cu \,dt$. Then we have that the limit
\begin{align}
\mathpzc{R}^0_\lambda u : = \lim_{c \to 0} \mathpzc{R}_\lambda^cu
\end{align}
exists for any $u \in C_0(\mathcal{S})$ and $  \frac{\lambda}{f(\lambda, \cdot)}\l f(\lambda, \cdot)I - G^0 \r \mathpzc{R}_\lambda^0=  I$.
\end{te}
\begin{proof}
First observe that by rearranging \eqref{rearranged} we have for any $u \in C_b \l \mathcal{S} \r$ that
\begin{align}
\frac{\lambda}{f(\lambda, \cdot)} \l f(\lambda, \cdot) - G^c \r  \mathpzc{R}_\lambda^cu \, = \, u.
\end{align}
We show that the elements $\mathpzc{R}_\lambda^{1/n}u$ for $n \in \mathbb{N}$ form a Cauchy sequence.
Let $h_n:= \frac{\lambda}{f(\lambda, \cdot)} \l f(\lambda, \cdot) I- G^{1/n} \r  w$ then we have that 
\begin{align}
h_n \to h= \frac{\lambda}{f(\lambda, \cdot)}  \l f(\lambda, \cdot) - G^0 \r  w.
\end{align}
Now note that
\begin{align}
\mathpzc{R}^{1/n}_\lambda h-\mathpzc{R}^{1/m}_\lambda h \, = \,  \mathpzc{R}^{1/m}_\lambda (h_m-h) + \mathpzc{R}^{1/n}_\lambda (h-h_n) + \l \mathpzc{R}^{1/n}_\lambda h_n - \mathpzc{R}^{1/m}_\lambda h_m \r.
\label{decres}
\end{align}
Since for any $n \in \mathbb{N}$ and $\lambda >0$ we have that $|| \mathpzc{R}_\lambda^{1/n}|| \leq 1/\Re\lambda$, the first two terms go to zero as $m,n \to \infty$ while the last term in \eqref{decres} is clearly zero. Therefore the limit exists for functions 
\begin{align}
u \in \frac{\lambda}{f(\lambda, \cdot)}  \l f(\lambda, \cdot) -G^0 \r  \text{Dom}(G^0) = \l \lambda - \frac{\lambda}{f(\lambda, \cdot)} G^0 \r \text{Dom}(G^0).
\end{align}
Since $G_0$ is a (dissipative) generator of a contraction semigroup we have (e.g. \cite[Thm 3.4.5]{abhn}) $\l \lambda -G^0 \r \text{Dom}(G^0) = C_0(\mathcal{S})$.
To conclude the proof observe that
\begin{align}
\frac{\lambda}{f(\lambda, \cdot)} \l f(\lambda, \cdot) - G^0 \r   \mathpzc{R}_\lambda^0 w \, = \, & \lim_{n \to \infty}  \frac{\lambda}{f(\lambda, \cdot)}  \l  f(\lambda, \cdot) - G^{1/n} \r  \mathpzc{R}_\lambda^{1/n} w \, = \, w.
\end{align}
\end{proof}
\begin{os} \normalfont
We remark that the limit $\mathpzc{R}_\lambda^0u$ obtained in Theorem \ref{telimit} satisfies, for $u \in \text{Dom}(G^0)$
\begin{align}
f(\lambda, \cdot)\mathpzc{R}_\lambda^0u- \lambda^{-1}f(\lambda, \cdot) u = G^0 \mathpzc{R}^c_\lambda u.
\label{inve}
\end{align}
Hence if one further assumes that the family $q^c(t):=\l \Pi_t^cu \r_{c \geq 0}$ has a limit then Theorem \ref{telimit} implies the convergence of $q^c(t)$ to the solution of the variable order equation 
\begin{align}
\mathcal{D}_t^{\bm{\cdot}}q^0(t) = G^0 q^0(t), \qquad q^0(0) = u \in \text{Dom}(G^0).
\end{align}
obtained by inverting ($\lambda \mapsto t$) the resolvent equation \eqref{inve}.
The equation \eqref{inve} plays the role of the Kolmogorov's equation for a limit process when such a process exists. This process indeed exists and still is a semi-Markov process, for example, when $X^c$ is a continuous time random walk with state space in $\mathbb{R}^d$ defined as in \cite[Section 2]{meerstra} and hence $c$ represents here a scale parameter. In this case the limit process is a semi-Markov process in the sense of Gihman and Skorohod (see Harlamov \cite[section 3.12, p. 76]{harlamov}), i.e., the limit process $X^0(t)$ is such that $\l X^0(t), \gamma^0(t) \r$ is a Markov process.
\end{os}
\subsubsection{Convergence to Brownian motion and L\'evy type processes} \label{convbm} Suppose that the embedded chain $X_n$ runs on $\mathbb{R}^d$ with transition probabilities given by
\begin{align}
h(x,dy) \, = \, \frac{1}{2d} \sum_{i=1}^d \l \delta_{x+e_i}(dy)+\delta_{x-e_i}(dy) \r.
\label{ciao}
\end{align}
Now if we set the transition probabilities of $X_n^c$ as in \eqref{ciao} with jump's height $c \geq 0$ we get
\begin{align}
h^c(x, dy) \, = \, \frac{1}{2d} \sum_{i=1}^d \l \delta_{x+ce_i}(dy)+\delta_{x-ce_i}(dy) \r.
\label{ciaoc}
\end{align}
Now suppose that the distribution of the waiting times $J_i$ is exponential with parameter $1/c^2$. By collecting all pieces together we obtain the corresponding Markov process $M^c(t)$ generated by
\begin{align}
\l G^cu \r (x) \, = \, & \frac{1}{2dc^2} \sum_{i=1}^d \l u(x+ce_i)+u(x-ce_i)-2u(x) \r.
 \end{align}
By letting $c \to 0$ we have, for functions $C^2 \l \mathbb{R}^d \r$, that $G^cu \to \frac{1}{2} \Delta u$ which is the generator of the strongly continuous heat semigroup on $C_0(\mathcal{S})$. Now suppose that the semi-Markov process $X^c(t)$ has the same embedded chain $X_n^c$ and has waiting times $J_i^c$ between jumps as in Theorem \ref{tetimech} having c.d.f. $\overline{F}_x(t) = \mathds{E}^z e^{-(1/c^2)L^{(x)}(t)}$ so that
Theorem \ref{telimfacile} holds and yields
\begin{align}
\int_0^\infty e^{-\lambda t} \mathds{E}^x u(X^c (t)) dt \to \widetilde{g}^0(\lambda,x)
\end{align}
where $\widetilde{g}^0 (\lambda,x)$ is the Laplace transform of the solution to
\begin{align}
\mathfrak{D}_t (x) g^0(t,x) \, = \, \frac{1}{2} \Delta g^0(t, x),
\end{align}
where
\begin{align}
\mathfrak{D}_t (x) g^0(t,x) \, = \, \frac{d}{dt} \int_0^t g^0(s, x) \, \bar{\nu}(t-s,x) \, ds \, - \, \bar{\nu}(t,x) g^0(0,x).
\end{align}

If instead the transition probabilities of the chain $X_n^c$ are
\begin{align}
h^c(x, d(y-x)) \, = \, \frac{\nu(x, d(y-x))}{\nu(x, \mathbb{R}^d \backslash B_c(0) )}  \mathds{1}_{ \left[ \mathbb{R}^d \backslash B_c(0) \right]}(y-x)
\end{align}
where $B_c(0) = \ll x \in \mathbb{R}^d : |x| \leq c \rr$ and $\mathcal{B}\l \mathbb{R}^d \r \ni B \mapsto \nu(x, B)$ is a $\sigma$-finite measure on $\mathbb{R}^d$ such that for any $x$
\begin{align}
\int_{\mathbb{R}^d} \l |y-x|^2 \wedge 1 \r \nu(x, d(y-x)) < \infty
\end{align}
and if the parameter of the exponential r.v. is $\theta^c(x) = \nu(x, \mathbb{R}^d \backslash B_c(0) )$ we obtain a Markov process $M^c(t)$ generated by
\begin{align}
\l G^cu \r (x) \, = \, \int_{\mathbb{R}^d\backslash B_c(0)} \l u(y) -u(x) \r \nu(x, d(y-x))
\end{align}
and by letting $c \to 0$ we obtain the L\'evy type generator \cite[formula (8.51), p. 366]{kolokoltsov}
\begin{align}
\l G^0 u \r (x) \, = \,  \int_{\mathbb{R}^d} \l u(y) -u(x) \r \nu(x, d(y-x)).
\end{align}
having domain on a subset of $\mathcal{C}_b(\mathcal{S})$.
Hence by using again Theorem \ref{telimfacile} we have that a semi-Markov process with the same embedded chain and whose waiting times $J_i^c$ have c.d.f.
\begin{align}
1-\overline{F}_x(t) \, = \, 1-\mathds{E}e^{-\theta^c(x) L^{(x)}(t)}
\end{align}
is such that
\begin{align}
\int_0^\infty e^{-\lambda t} \mathds{E}^xu(X^c(t))dt \to \widetilde{g}^0(\lambda,x)
\end{align}
where $\widetilde{g}^0(\lambda,x)$ is the Laplace transform of the solution to
\begin{align}
\mathfrak{D}_t(x) g(t, x) \, = \, \int_{\mathbb{R}^d} \l g(t,y) -g(t,x) \r \nu(x, d(y-x)).
\end{align}

\subsection{Convergence to the fractional equation}
\label{seclimit}
In this section we focus on the situation in which the waiting times $J_i$ have infinite mean. This is obtained here by assuming that the Bernstein functions $f(\lambda, x)$ defining the exponent $\Pi$ are regularly varying at zero with index $\alpha(x) \in (0,1)$. We recall that a function $f(\lambda)$ vary regularly at zero if $g(\lambda):=f(1/\lambda)$ vary regularly at infinity, i.e., 
\begin{align}
\lim_{\lambda \to \infty} \frac{g(c\lambda)}{g(\lambda)} = c^\alpha, \qquad \alpha  \geq 0.
\end{align}
By Karamata's characterization Theorem then we know that every reguarly varying function can be written in form $g(\lambda) = \lambda^\alpha L(\lambda)$ where $L(\lambda)$ is a slowly varying function, hence $L$ is such that
\begin{align}
\lim_{\lambda \to \infty} \frac{L(c\lambda)}{L(\lambda)} = 1.
\end{align}
This assumption on the Bernstein function $f(\lambda, x)$ implies by \cite[formula (2.32)]{meertoa} that
\begin{align}
\int_0^\infty \overline{F}_x(t) \, dt \, = \, \int_0^\infty \mathds{E}^ze^{-\theta(x)L^{(x)}(t)} \, dt \, = \, \infty.
\end{align}
Hence the waiting times $J_i$ with c.d.f. $1-\overline{F}_x(t)$ are such that $\mathds{E}J_i = \infty$.
In the following theorem we assume again that the generator $G^c$ converges to a generator $G^0$. Under the additional assumption on regular variation described above we study the limit of the corresponding semi-Markov processes.
\begin{te}
\label{telimit}
Let $M^c$ be a family of Markov processes associated with the semigroups $P_{t}^c$ on the space $C_b \l \mathcal{S} \r$ having generators $G^cu$ which is given by $(G^cu)(\cdot):=\theta^c(\cdot) \int_S \l u(y)-u(\cdot) \r h^c(\cdot,dy)$. Let $G^0$ be a (dissipative) generator of a semigroup on the Banach space $C_0(\mathcal{S})$ equipped with the sup-norm and assume that $c^{-1}G^cu \to G^0u$ for $u \in \text{Dom}(G^0)$. Let $f(\lambda, x)$ be regularly varying at zero with index $\alpha(x) \in (0,1)$ for any $x \in \mathcal{S}$. Now denote $P_t^{f(c,\cdot)}$ the semigroup generated by $G^{f(c,\cdot)}u:= \theta^{f(c, \cdot)}(\cdot)\int_{\mathcal{S}}\l u(y) - u(\cdot) \r h^{f(c, \cdot)}(\cdot,dy)$ and denote the corresponding Markov process as $M^{f(c,\cdot)}$. Let $\l X^c(t) \r_{c \geq 0}$ be a family of stepped semi-Markov processes defined as in Section \ref{defsemi} each one of which satisfies the assumptions of Theorem \ref{tetimech} with respect to the Markov process $M^{f(c,x)}$ and hence denote 
\begin{align}
\l \Pi_t^cu \r (x) := \mathds{E}^x u\l X^c(t) \r \, = \, \mathds{E}^xu \l M^{f(c,x)} \l L^\Pi (t) \r \r.
\end{align}
Let $\mathpzc{R}_\lambda^cu:=\int_0^\infty e^{-\lambda t} \Pi_{t/c}u\, dt$. Then we have that the limit
\begin{align}
\mathpzc{R}^0_\lambda u:=\lim_{c \to 0}\mathpzc{R}^{f(c, \cdot)}_\lambda u 
\end{align}
exists for any $u \in C_0(\mathcal{S})$ and $ \lambda^{1-\alpha(\cdot)}\l \lambda^{\alpha(\cdot)}I-G^0 \r\mathpzc{R}_\lambda^0= I$.
\end{te}
\begin{proof}
First observe that
\begin{align}
\mathpzc{R}_\lambda^c u \, = \, c\int_0^\infty e^{-\lambda ct} \Pi_tu \, dt
\end{align}
and hence as in the proof of theorem \ref{telimfacile} we find, by rearranging \eqref{rearranged}, that for any $u \in C_b \l \mathcal{S} \r$
\begin{align}
\frac{\lambda}{f(c\lambda, \cdot)} \l f(c\lambda, \cdot) - G^{f(c, \cdot)} \r  \mathpzc{R}_\lambda^cu \, = \, u.
\end{align}
Let $w \in \text{Dom}(G^0)$ and use the fact that $f(\lambda, x)$ is regularly varying at zero to say that $f(\lambda, x) = \lambda^{\alpha(x)}L(\lambda, x)$ where $\lambda \mapsto L(\lambda, x)$ is slowly varying. Hence, as $c \to 0$,
\begin{align}
\lim_{c \to 0}\frac{\lambda}{f(c\lambda, \cdot)} \l f(c\lambda, \cdot) - G^{f(c, \cdot)} \r  w \, = \, &\lim_{c \to 0} \frac{\lambda^{1-\alpha(\cdot)}}{c^{\alpha (\cdot)}L(\lambda c, \cdot)} \l (\lambda c)^{\alpha(\cdot)}L(\lambda c, \cdot)-G^{f(c, \cdot)} \r w \notag \\
= \,&\lim_{c \to 0} \frac{\lambda^{1-\alpha(\cdot)}}{f(c, \cdot)} \l (\lambda )^{\alpha(\cdot)}f( c, \cdot)-G^{f(c, \cdot)} \r w \notag \\
= \, &\lambda^{1-\alpha(\cdot)} \lim_{c \to 0} \l \lambda^{\alpha (\cdot)} - \frac{1}{f(c, \cdot)} G^{f(c, \cdot)} \r w \notag \\
= \, & \lambda^{1-\alpha(\cdot)}  \l \lambda^{\alpha (\cdot)} - G^0 \r w.
\label{525}
\end{align}
The elements $\mathpzc{R}_\lambda^{1/n}u$ for $n \in \mathbb{N}$ form a Cauchy sequence. This can be proved as in the proof of Theorem \ref{telimfacile}.
Let $h_n:= \frac{\lambda}{f(\lambda/n, \cdot)} \l f(\lambda/n, \cdot) I- G^{f \l 1/n, \cdot \r} \r  w$ then we have by \eqref{525}
\begin{align}
h_n \to h= \lambda^{1-\alpha(\cdot)}\l \lambda^{\alpha (\cdot)} - G^0 \r w.
\end{align}
Now note that
\begin{align}
\mathpzc{R}^{1/n}_\lambda h-\mathpzc{R}^{1/m}_\lambda h \, = \,  \mathpzc{R}^{1/m}_\lambda (h_m-h) + \mathpzc{R}^{1/n}_\lambda (h-h_n) + \l \mathpzc{R}^{1/n}_\lambda h_n - \mathpzc{R}^{1/m}_\lambda h_m \r.
\label{527}
\end{align}
Since for any $n \in \mathbb{N}$ and $\lambda >0$ we have that $|| \mathpzc{R}_\lambda^{1/n}|| \leq 1/\Re\lambda$, the first two terms go to zero as $m,n \to \infty$ while the last term in \eqref{527} is clearly zero. Therefore the limit exists for functions 
\begin{align}
u \in \lambda^{1-\alpha(\cdot)}\l \lambda^{\alpha (\cdot)} - G^0 \r  \text{Dom}(G^0) = \l \lambda - \lambda^{1-\alpha(\cdot)} G^0 \r \text{Dom}(G^0).
\end{align}
Since $G_0$ is dissipative we have (e.g. \cite[Corollary 3.4.6]{abhn}) $\l \lambda -G^0 \r \text{Dom}(G^0) = \mathbb{C} \l \mathcal{S} \r$.
To conclude the proof observe that
\begin{align}
\lambda^{1-\alpha(\cdot)} \l \lambda^{\alpha (\cdot)} - G^0 \r   \mathpzc{R}_\lambda^0 w \, = \, & \lim_{n \to \infty}  \frac{\lambda}{f(\lambda/n,\cdot)}  \l  f(\lambda/n, \cdot) - G^{f\l 1/n, \cdot \r} \r  \mathpzc{R}_\lambda^{1/n} w \, = \, w.
\end{align}
\end{proof}
\begin{os} \normalfont
Note that the limit $\mathpzc{R}_\lambda^0u$ obtained in Theorem \ref{telimit} satisfies, for $u \in \text{Dom}(G^0)$
\begin{align}
\lambda^{\alpha (\cdot)}\mathpzc{R}_\lambda^0u- \lambda^{\alpha (\cdot)-1} u = G^0 \mathpzc{R}^0_\lambda u.
\label{inverfrac}
\end{align}
By inverting Laplace transform in \eqref{inverfrac} $(\lambda \mapsto t)$ we get the variable order fractional equation
\begin{align}
\frac{d^{\alpha(\cdot)}}{dt^{\alpha(\cdot)}} q^0(t) = G^0 q^0(t), \qquad q^0(t) = u.
\label{fracvareq}
\end{align}
Hence if we additionally assume that the family $q^c(t):=\l \Pi_t^cu \r_{c \geq 0}$ has a limit then Theorem \ref{telimit} implies the convergence of $q^c(t)$ to the solution of the fractional variable order equation \eqref{fracvareq}.

As stated in \cite{kurtz} the asymptotic behaviour of a semi-Markov process is Markovian when the holding times have finite mean. Here the limit process cannot be Markovian since it is governed by a fractional equation. Hence the asymptotic Markovian behaviour is lost due to such heavy tailed waiting times.
\end{os}

\subsubsection{The limit of the Poisson process and the inverse stable subordinator}
Consider the strongly continuous Poisson semigroup on the space $C_b \l \mathbb{R}^d \r$ defined as
\begin{align}
(P_t^c u)(x) \, = \, \sum_{j=0}^\infty u(x+lj) \frac{(\theta t)^j}{j!}e^{-\theta t}
\end{align}
for $l \in \mathbb{R}^d$ with $|l|=1$. This correspond to the Poisson process, say $N^l(t)$, with intensity $\theta$ and jump height $l \in \mathbb{R}^d$ introduced in Example \ref{expoieq}. Now let $P_t^c$ be the family of Poisson semigroups
\begin{align}
(P_t^cu)(x) \, = \, \sum_{j=0}^\infty u(x+clj) \frac{(\theta t)^j}{j!}e^{-\theta t}
\end{align}
which corresponds to the Poisson process $N^{cl}$ with jump height $cl$, $l \in \mathbb{R}^d$, $c >0$ and consider the operator $P_{t/c}^c$ which is still a Poisson semigroup and has generator
\begin{align}
(G^c u)(x) \, : = \, \theta \frac{u(x+cl)-u(x)}{c}.
\end{align}
Hence by letting $c \to 0$ we have that
\begin{align}
G^c u \to  G^0 u \, = \, \theta \nabla_l u, \qquad u \in C^1_l(\mathbb{R}^d).
\end{align}
Let $X_n^c$ be a Markov chain with transition probabilities $h^{f(c,x)}(x,dy) = \delta_{x+f(c,x)l}$ and consider the semigroup $P_t^{f(c,x)}$ generated by
\begin{align}
(G^{f(c,x)}u)(x) \, = \, \theta \l u(x+f(c,x)l)-u(x)\r
\end{align}
and let $N^{f(c,x)l}$ be the associated Markov chain. Now let $\mathpzc{J}_i$ be i.i.d. independent r.v.'s with distribution
\begin{align}
\overline{F}_y(t) \, = \, P^x \l \mathpzc{J}_i >t \mid X_i = y, X_{i+1} = z \r
\end{align}
and define
\begin{align}
\mathpzc{N}^c(t) \, = \, X_n^c, \qquad \sum_{i=0}^n \mathpzc{J}_i \leq t < \sum_{i=0}^{n+1} \mathpzc{J}_i,
\end{align}
and denote
\begin{align}
(\Pi_t^cu)(x) \, := \, \mathds{E}^x u \l \mathpzc{N}^c(t) \r \, = \, \mathds{E}^x u \l N^{f(c,x)l} \l L^\Pi (t) \r \r
\end{align}
and assume that $\widetilde{\overline{F}}_x(\lambda)$ vary regularly at zero with index $\alpha(x)-1$ for $\alpha(x) \in (0,1)$, for any $x \in \mathbb{R}^d$. Use this and \eqref{easyto} to check that $\lambda \mapsto f(\lambda, x)$ vary regularly at zero with index $\alpha(x) \in (0,1)$ for any $x \in \mathcal{S}$. 
Now apply Theorem \ref{telimit} to say that for any $x \in \mathcal{S}$,
\begin{align}
\int_0^\infty e^{-\lambda t} (\Pi_{t/c}^c u)(x) \, dt \, = \, \int_0^\infty e^{-\lambda t} \, \mathds{E}^x u \l \mathpzc{N}^c(t/c) \r \, dt \, \stackrel{c \to 0}{\longrightarrow} \,  \widetilde{q}^0 (x,\lambda)
\label{526}
\end{align}
where $\widetilde{q}^0 (\lambda)$ is the Laplace transform of the solution to
\begin{align}
\frac{d^{\alpha(\cdot)}}{dt^{\alpha(\cdot)}} q^0(t) \, = \, \theta \nabla_l q^0(t), \qquad q^0(0) = u \in C^1_l (\mathbb{R}^d).
\label{fractrasl}
\end{align}

\subsubsection{Convergence to the fractional diffusion equation and to L\'evy type processes}
Let $X_n$ be the embedded chain with transition probabilities given in \eqref{ciaoc}. Assume that the waiting times of the Markov process $M^c(t)$ having $X_n$ as embedded chain have exponential distribution with parameter $\theta^c(x)=1/c$. Hence the Markov process $M^c(t)$ has generator
\begin{align}
\l G^cu \r (x) \, = \, & \frac{1}{2dc^2} \sum_{i=1}^d \l u(x+ce_i)+u(x-ce_i)-2u(x) \r.
\end{align}
Hence we have as in Section \ref{convbm} that $G^cu \to (1/2)\Delta u$ as $ c \to 0$ for $u \in C^2(\mathbb{R}^d)$ which is the generator of the strongly continuous heat semigroup on $C_0(\mathcal{S})$. Now consider the process $M^{f(c,\cdot)}(t)$ which is obtained by setting the transition probabilities of the embedded chain $X_n$
\begin{align}
h^{f(c,x)}(x, dy) \, = \, \frac{1}{2d} \sum_{i=1}^d \l \delta_{x+f(c,x)e_i}(dy) -  \delta_{x-f(c,x)e_i}(dy) \r
\end{align}
and the paramaters of the exponential r.v.'s equal to $\theta^c(x)=1/f(c,x)$. Assume that $f(\lambda, x)$ varies regularly at zero with index $\alpha (x) \in (0,1)$ and define $X^c(t)$ as a semi-Markov process which satisfies the condition in Theorem \ref{tetimech}. Hence we can apply Theorem \ref{telimit} to say that
\begin{align}
\int_0^\infty e^{-\lambda t} \mathds{E}^xu(X^c(t/c)) dt \to \widetilde{g}^0 (\lambda, x)
\end{align}
where $\widetilde{g}^0(\lambda)$ is the Laplace transform of the solution to
\begin{align}
\frac{d^{\alpha (\cdot)}}{dt^{\alpha(\cdot)}}g^0(t) \, = \, \frac{1}{2} \Delta g^0(t), \qquad g^0(0) = u \in C^2 \l \mathcal{S} \r.
\end{align}

Consider again the transition probabilities of the chain $X_n^c$ as in Section \ref{convbm} given by
\begin{align}
h^c(x, d(y-x)) \, = \, \frac{\nu(x, d(y-x))}{\nu(x, \mathbb{R}^d \backslash B_c(0) )}  \mathds{1}_{ \left[ \mathbb{R}^d \backslash B_c(0) \right]}(y-x)
\end{align}
which yields to a Markov process $M^c(t)$ with semigroup $P_t^c$ converging to the L\'evy type semigroup generated by 
\begin{align}
\l G^0 u \r (x) \, = \,  \int_{\mathbb{R}^d} \l u(y) -u(x) \r \nu(x, d(y-x)).
\end{align}
In order to apply Theorem \ref{telimit} we need to consider the Markov process $M^{f(c,x)}(t)$ having transition probabilities
\begin{align}
h^{f(c,x)}(x, d(y-x)) \, = \, \frac{\nu(x, d(y-x))}{\nu\l x, \mathbb{R}^d \backslash B_{f(c,x)}(0) \r}  \mathds{1}_{ \left[ \mathbb{R}^d \backslash B_{f(c,x)}(0) \right]}(y-x)
\end{align}
where $B_{f(c,x)}(0) = \ll y \in \mathbb{R}^d : |y| \leq f(c,x), x \in \mathbb{R}^d \rr$. Now let $X^c(t)$ be a family of semi-Markov process satisfying the hypothesys in Theorem \ref{tetimech} with respect to $M^{f(c,x)}(t)$ and use Theorem \ref{telimit} to say that
\begin{align}
\int_0^\infty e^{-\lambda t} \mathds{E}^xu(X^c(t)) dt \, \to \, \widetilde{g}^0(\lambda,x) 
\end{align}
as $c \to 0$, where $\widetilde{g}^0(\lambda,x)$ is the Laplace transform of $g^0(t,x)$ which satisfies
\begin{align}
\frac{d^{\alpha(x)}}{dt^{\alpha(x)}} g(t, x) \, = \, \int_{\mathbb{R}^d}  \l g(t,y) -g(t,x) \r \nu(x, d(y-x)).
\end{align}

\section{Some remarks on countable state spaces}
For the sake of intuition we present here some results of previous sections when the state space $\mathcal{S}$ is countable. This situation is often useful in applications (e.g. \cite{jannsen} and also \cite{G15, raberto} for some recent examples). The results here are a consequence of the ones obtained above and hence we present here only the points in which the discussion becomes easier. Since the state space is countable we define the transition matrix $\l P_t \r_{ij} := P^i \l M(t) = j \r$, which is a semigroup of linear operators acting on vectors $\underline{u}$ generated by
\begin{align}
G \,: = \, \Theta \l H - 1 \r,
\label{genM}
\end{align}
where $\Theta = \text{diag}(\theta_1, \theta_2, \cdots)$ and $H$ is the transition matrix of the embedded Markov chain $X_n$. 
Hence we have that
\begin{align}
\frac{d}{dt} P_t \underline{u} = G P_t\underline{u} = P_t G\underline{u}.
\label{keq}
\end{align} 
Consider now the matrix $\Pi_t$ given by
\begin{align}
\pi_{i,j}(t) \, :=  \l \Pi_t \r_{ij} :=\,  P^i \l X(t) = j \mid \gamma^X(0) = 0  \r
\end{align}
which is equivalent, for any $x \in \mathcal{S}$, to
\begin{align}
\l \Pi_t \r_{ij} = P^x \l X(t+\tau) = j \mid X(\tau) = i, \gamma^X(\tau) = 0  \r
\end{align}
in view of homogeneity \eqref{homo}. The backward equation has the following form.
\begin{coro}
\label{coroclasseq}
Let $\l \Pi_t \r_{ij} = \pi_{i,j}= P^i \l X(t) = j \r$ and consider as initial datum a vector $\underline{u} = \left[ u(1), u(2), \dots, u(i), \dots \right]^\prime$, $i=1, 2, \dots$ where $u$ is $C_b \l \mathcal{S} \r$ and satisfies the assumptions in Proposition \ref{tediff}. The mapping $ t \mapsto \Pi_tu$ solves, for $t \geq 0$,
\begin{align}
\mathfrak{D}_t^\cdot q(t) \, = \, Gq(t), \qquad q(0) = \underline{u}.
\label{classiceq}
\end{align}
\end{coro}
\begin{proof}
This is a Corollary of Theorem \ref{teclassiceq}. We provide here an heuristic proof. Use \cite[Chapter 10, formula (5.5)]{cinlar} to say that $\pi_{i,j}(t)$ satisfies the backward renewal equation
\begin{align}
\pi_{i,j}(t) \, = \, \overline{F}_{i}(t) \delta_{ij} \, + \, \sum_{k \in \mathcal{S}} \int_0^t h_{ik} \,  \mathpzc{f}_i (s) \, \pi_{k,j}(t-s) \,  ds
\label{410}
\end{align}
where $f$ is a density of $F$. Since $\overline{F}_i(t) = \mathds{E}^x e^{-\theta_iL^i(t)}$ we recall that
\begin{align}
\mathcal{L} \left[ \overline{F}_i(\cdot) \right] (\lambda)\, = \, \frac{f(\lambda, i)}{\lambda} \frac{1}{\theta_i+f(\lambda, i)}
\end{align}
and therefore by the convolution theorem for Laplace transform we can take the Laplace transform in \eqref{410} to write
\begin{align}
\widetilde{\pi}_{i,j}(\lambda) \, = \, \frac{f(\lambda, i)}{\lambda} \frac{1}{\theta_i+f(\lambda, i)} \delta_{ij}+ \sum_{k \in \mathcal{S}} \frac{\theta_i}{\theta_i+f(\lambda, i)} h_{ik}\widetilde{\pi}_{k,j}(\lambda).
\label{rearr}
\end{align}
By rearranging \eqref{rearr} we can write
\begin{align}
f(\lambda, i) \widetilde{\pi}_{i,j}(\lambda) - \lambda^{-1}f(\lambda, i) \pi_{i,j}(0) \, = \, -\theta_i \widetilde{\pi}_{i,j}(\lambda) + \sum_{k \in \mathcal{S}} \theta_i h_{ik} \widetilde{\pi}_{k,j}(\lambda).
\label{rearr2}
\end{align}
By using \cite[Proposition 2.7 and Lemma 2.5]{toaldopota} we can invert the Laplace transform and we have that
\begin{align}
\frac{d}{dt} \int_0^t \pi_{i,j} (s) \, \bar{\nu}(t-s,i) ds \, - \delta_{ij} \bar{\nu}(t,i) \, = \, -\theta_i \pi_{i,j}(t) + \sum_{k \in \mathcal{S}} \theta_i h_{ik} \pi_{k,j}(t).
\label{414}
\end{align}
Let $g_{ij}$ be the elements of the matrix $G$. By definition \eqref{genM} we have that
\begin{align}
g_{ij} \, = \, \theta_i \l h_{ij}-\delta_{ij} \r
\label{elg}
\end{align}
and hence \eqref{414} reduces to
\begin{align}
\frac{d}{dt} \int_0^t \pi_{i,j} (s) \, \bar{\nu}(t-s,i) ds \,- \delta_{ij} \bar{\nu}(t,i) \, = \, \sum_{k \in \mathcal{S}} g_{ik} \, \pi_{k,j}(t).
\end{align}
\end{proof}
The version of the equation in Theorem \ref{teeqevol} in a countable space is given in the following Corollary.
\begin{coro}
Under the same assumptions of Corollary \eqref{coroclasseq} we have that $q(t)$ satisfies the evolutionary Cauchy problem
\begin{align}
\frac{d}{dt} q(t) \, = \, \mathcal{D}_t^\star  \, G \, q(t), \qquad q(0) = u.
\end{align}
\end{coro}
\begin{proof}
This is a consequence of Theorem \ref{teeqevol}. To give an heuristic proof as in Corollary \ref{coroclasseq} it is sufficient to rearrange \eqref{rearr2} as
\begin{align}
\lambda \widetilde{\pi}_{i,j}(\lambda) -  \pi_{i,j}(0) \, = \, -\frac{\lambda}{f(\lambda, i)}\l\theta_i \widetilde{\pi}_{i,j}(\lambda) + \sum_{k \in \mathcal{S}} \theta_i h_{ik} \widetilde{\pi}_{k,j}(\lambda)\r.
\end{align}
By inverting the Laplace transform and using \eqref{elg} we obtain
\begin{align}
\frac{d}{dt}\pi_{i,j}(t) \, = \, \frac{d}{dt} \int_0^t \sum_{k \in \mathcal{S}} g_{ik} \, \pi_{k,j}(s) u^f(t-s,i)ds.
\end{align}
\end{proof}

\section{Auxiliary results}
\br{We collect here two technical results used in the paper.}
\begin{lem} 
\label{tecommon}
Let $\sigma^f$ be a strictly increasing subordinator with Laplace exponent $f(\lambda)$ and let $L^f(t)$ be the hitting time of $\sigma^f(t)$. The function $t \mapsto \mathds{E}^xe^{-\br{\theta}L^f(t)}$, $\theta >0$, is completely monotone if and only if $f(\lambda)$ is a complete Bernstein function.
\end{lem}
\begin{proof}
Since we assume that $\sigma^f(t)$ is strictly increasing it must be true that $b>0$ and $\nu(0, \infty) < \infty$, or $b=0$ and $\nu(0, \infty) = \infty$, or $b>0$ and $\nu(0, \infty) = \infty$.
In \cite[Theorem 2.1]{meertoa} it is proved that if $b=0$ and $\nu(0, \infty)=\infty$ the function $t \mapsto \mathds{E}^xe^{-\br{\theta}L^f(t)}$ is completely monotone if and only if the tail $t \mapsto \bar{\nu}(t)$ is a completely monotone function. This proves the result when $b=0$ and $\nu(0, \infty)=\infty$ since the complete monotonicity of $t \mapsto \bar{\nu}(t)$ implies that the L\'evy measure $\nu(\cdot)$ has a completely monotone density and therefore $\lambda \mapsto f(\lambda)$ is a complete Bernstein function (e.g. the discussion in \cite[p. 91]{pottheory}). When $b>0$ and $\nu(0, \infty) < \infty$ the proof can be very similar and we write here the basic facts. 

First we prove the direct statement. Since $t \mapsto \sigma^f (t)$ is strictly increasing then $t \mapsto L^f(t)$ are a.s. continuous functions. Hence $L^f(t)$ is a.s. continuous and also in distribution. The function $t \mapsto \mathds{E}^xe^{-\br{\theta}L^f(t)}$ is therefore continuous by \cite[Theorem 4, p. 431]{Feller} and has Laplace transform \cite[Corollary 3.5]{meertri}
\begin{align}
\int_0^\infty e^{-\lambda t} \mathds{E}^xe^{-L^f(t)}dt \, = \, \frac{f(\lambda)}{\lambda} \frac{1}{\br{\theta}+f(\lambda)}.
\label{unlapl}
\end{align}
Since
\begin{align}
\lambda \mapsto \varphi (\lambda) =  \frac{\lambda}{\br{\theta}+\lambda}
\end{align}
is a complete Bernstein function then by \cite[Theorem 7.6]{librobern}
\begin{align}
\lambda \mapsto \l \varphi \circ f \r (\lambda) \, = \,  \frac{f(\lambda)}{\br{\theta}+f(\lambda)}
\end{align}
is a complete Bernstein function. Therefore there exists a triple $\l \mathpzc{a}, \mathpzc{b}, \mathpzc{v} \r$ such that
\begin{align}
\lambda \mapsto \l \varphi \circ f \r (\lambda) \, = \, \mathpzc{a} + \mathpzc{b}\lambda + \int_0^\infty \l 1-e^{-\lambda s} \r \mathpzc{v}(ds)
\end{align}
where the L\'evy measure $\mathpzc{v}(\cdot)$ has a completely monotone density and the tail $t \mapsto \bar{\mathpzc{\nu}}(t) = \mathpzc{a}+ \mathpzc{v}(t, \infty)$ is a completely monotone function \cite[p. 91]{pottheory}. After an integration by parts we can also write
\begin{align}
\lambda \mapsto \frac{1}{\lambda} \l \varphi \circ f \r (\lambda) \, = \,   \int_0^\infty e^{-\lambda s} \l \mathpzc{b}\lambda+ \bar{\mathpzc{v}}(s) \r ds
\end{align}
and since $s \mapsto \mathpzc{b}\lambda + \bar{\mathpzc{\nu}}(s)$ is completely monotone (and obviously continuous) we can use \eqref{unlapl} and the unicity of Laplace transform to say that $t \mapsto \mathds{E}^xe^{-\br{\theta}L^f(t)}$ is completely monotone. Now we prove the converse statement and therefore assume that
\begin{align}
t \mapsto \mathds{E}^xe^{-\br{\theta}L^f(t)} \, = \, \int_0^\infty e^{-st} \mathpzc{m}(ds)
\label{convstat}
\end{align}
for some measure $\mathpzc{m}(\cdot)$. Then by using \eqref{unlapl} and \eqref{convstat} we can write
\begin{align}
\frac{f(\lambda)}{\lambda} \frac{1}{\br{\theta}+f(\lambda)} \, = \, \int_0^\infty e^{-\lambda t} \mathds{E}^xe^{-L^f(t)} \, dt \, = \, \int_0^\infty \frac{1}{\lambda + s} \mathpzc{m}(ds)
\end{align}
and thus $\lambda \mapsto \lambda^{-1}f(\lambda)\big/ \l \br{\theta}+f(\lambda) \r$ is a Stieltjes function \cite[Definition 2.1]{librobern} if 
\begin{align}
\int_0^\infty \frac{1}{1+s} \mathpzc{m}(ds) \, < \, \infty.
\label{416}
\end{align}
To verify \eqref{416}, observe that the subordinator $\sigma^f(t)$ is a.s. increasing (and non-negative) and therefore, a.s., $L^f(0) =0$. Hence $\mathds{E}^xe^{-\br{\theta}L^f(0)}=1$, a.s.. Applying this to \eqref{convstat} allows us to write $\int_0^\infty \mathpzc{m}(ds) < \infty$ and thus \eqref{416} is true.

Then $\lambda \mapsto f(\lambda) / (1+f(\lambda))$ is a complete Bernstein function by \cite[Theorem 6.2]{librobern} and therefore $\lambda \mapsto  (1/f(\lambda))+1$ is a Stieltjes function by \cite[Theorem 7.3]{librobern} (and therefore also $1/f(\lambda)$). Another application of \cite[Theorem 7.3]{librobern} ensures that $\lambda \mapsto f(\lambda)$ is a complete Bernstein function. This concludes the proof.
\end{proof}
\begin{lem}
\label{te22}
For every non-negative Stieltjes function $\mathfrak{G} (\lambda)$ with representation 
\begin{align}
\mathfrak{G} (\lambda) \, = \, c\lambda^{-1}+\int_0^\infty \frac{1}{s+\lambda} \mathfrak{K}(ds) \quad \text{with} \quad 1-c \leq \mathfrak{K}(0, \infty) < \infty
\label{222}
\end{align}
for $0 \leq c < 1$, there exists a complete Bernstein function $f$ such that
\begin{align}
\mathfrak{G} \l \lambda \r \, = \,   \frac{1}{\lambda} \frac{f(\lambda)}{\br{\theta}+f(\lambda)},
\label{215}
\end{align}
\br{for some constant $\theta >0$.}
\end{lem}
\begin{proof}
By rearranging \eqref{215} we note that in order to prove \eqref{215} we can prove that
\begin{align}
f(\lambda)= \br{\theta}\frac{\lambda \mathfrak{G}(\lambda)}{1- \lambda \mathfrak{G}(\lambda)} 
\end{align}
is a complete Bernstein function for any Stieltjes function $\mathfrak{G}$ having representation \eqref{222}. \br{Since here $\theta$ is just a multiplicative constant we can prove the Theorem for $\theta =1$. Indeed if $f(\lambda)$ is a complete Bernstein function then also $\theta^{-1}f(\lambda)$ is, for any $\theta >0$.}
Since $\mathfrak{G}(\lambda)$ is a Stieltjes function we have that $\lambda \mathfrak{G}(\lambda)$ is a complete Bernstein function by \cite[Theorem 6.2]{librobern}. Then $1/( \lambda \mathfrak{G}(\lambda)) $ is Stieltjes by \cite[Theorem 7.3]{librobern} and therefore
\begin{align}
\frac{1}{\lambda \mathfrak{G}(\lambda)} \, = \, \frac{\mathfrak{a}}{\lambda} + \mathfrak{b} + \int_0^\infty \frac{1}{s+\lambda} \mathfrak{m}(ds)
\end{align}
for some constant $\mathfrak{a},\mathfrak{b}$ and a measure $\mathfrak{m}(\cdot)$. Note that
\begin{align}
\mathfrak{b} \, = \, \lim_{\lambda \to \infty} \frac{1}{\lambda \mathfrak{G}(\lambda)}
\end{align}
but 
\begin{align}
\lim_{\lambda \to \infty} \lambda \mathfrak{G}(\lambda) \, = \, c+ \lim_{\lambda \to \infty} \int_0^\infty \l 1-\frac{s}{s+\lambda}\r \mathfrak{K}(ds) \, = \, c+ \mathfrak{K}(0, \infty)
\label{218}
\end{align}
where we used the monotone convergence Theorem to move the limit inside the integral. Hence $\mathfrak{b}\geq 1$. Therefore the function
\begin{align}
\frac{1}{\lambda \mathfrak{G}(\lambda)}-1
\end{align}
is a (non-negative) Stieltjes function and thus by \cite[Theorem 7.3]{librobern} there exists a complete Bernstein function $f(\lambda)$ such that
\begin{align}
\frac{1}{f(\lambda)} \, = \, \frac{1}{\lambda \mathfrak{G}(\lambda)}-1 
\end{align}
and this proves that the function
\begin{align}
\lambda \mapsto f(\lambda) \, = \, \frac{\lambda \mathfrak{G}(\lambda)}{1-\lambda \mathfrak{G}(\lambda)}
\label{140}
\end{align}
is a complete Bernstein function.
\end{proof}

\vspace{1cm}
\end{document}